\documentclass[11pt]{article}
\usepackage{amsmath,amsthm,amsfonts,amssymb,amscd, amsxtra, mathrsfs}
\usepackage{url}
\usepackage[margin=2.7 cm,nohead]{geometry}
\usepackage{color}
\usepackage{float}
\usepackage{cuted}
\newtheorem{algorithm}{Algorithm}

\usepackage{graphicx}
\usepackage{subcaption}
\usepackage{caption}
\usepackage{pdflscape}
\usepackage{geometry}

 \usepackage{enumerate}
\usepackage{booktabs}

\usepackage{algorithm}
\usepackage{algpseudocode}
 
\newtheorem{theorem}{Theorem}
\newtheorem{lemma}[theorem]{Lemma}
\newtheorem{definition}{Definition}
\newtheorem{corollary}[theorem]{Corollary}
\newtheorem{proposition}[theorem]{Proposition}

\newtheorem{remark}{Remark}

\usepackage[margin=2.7 cm,nohead]{geometry}
\usepackage{color}
\DeclareMathOperator{\grad}{grad}

\usepackage{enumitem}

\DeclareMathOperator{\Exp}{Exp}
\DeclareMathOperator{\Log}{Log}

\usepackage{amsmath} % ya deberías tenerlo

\begin{document}
%%%%%%%%%%%%%%%%%%%%%%%%%%%%%%%%%%%%%%%%%%%%%%%%%%%%%%%%%%%%%
\title{Projected Gradient Method on Hadamard Manifolds}
\author{
 O. P. Ferreira\thanks{Institute of Mathematics and Statistics, Federal University of Goias, Avenida Esperan\c{c}a, s/n, Campus II,  Goi\^ania, GO - 74690-900, Brazil (E-mail: {\tt orizon@ufg.br}, {\tt maxlng@ufg.br}, {\tt amunozg@discente.ufg.br}, {\tt mauriciolouzeiro@ufg.br}). The works of M. L. N. Gon\c calves was supported in part by CNPq (Grant No.  403197/2025-2), FAPEG (Grant No. 202510267001610) and FAPESC (Grant No. 2024TR002238).}   \and
M. L. N. Gon\c calves \footnotemark[1]
\and
A. M. Gonz\'alez \footnotemark[1]
\and
M. S. Louzeiro  \footnotemark[1]
\and
}

\maketitle

%\vspace{.1cm}
\noindent

\noindent
{\bf Abstract:}
We study constrained smooth optimization problems on Hadamard manifolds with closed geodesically convex feasible sets. We analyze two projected gradient schemes: one with a constant stepsize and another with a backtracking line search. The constant-stepsize scheme is analyzed under the assumption that the objective function has a Lipschitz continuous Riemannian gradient, whereas the backtracking variant does not require this assumption to establish stationarity of accumulation points. For both schemes, we prove that every accumulation point of the generated sequence is first-order stationary under the respective assumptions, without requiring compactness of the feasible set; compactness is needed only to ensure the existence of accumulation points. When the objective function has a Lipschitz continuous Riemannian gradient, we derive iteration-complexity bounds of order \(O(1/\sqrt{N})\) for projection-based stationarity measures for both schemes, together with the corresponding \(\varepsilon\)-complexity estimates. For the backtracking scheme, the complexity analysis additionally requires the trial line-search stepsizes to be uniformly bounded away from zero. Under the same respective assumptions, the generated sequences are also asymptotically regular.   Finally, we illustrate the practical performance of the methods by solving constrained Karcher mean problems on the manifold of symmetric positive definite matrices.\\\\
\noindent 
{\bf Keywords:} Hadamard manifolds; projected gradient method; metric projection; convex set.

%%%%%%%%%%%%%%%%%%%%%%%%%%%%%%%%%%%%%%%
\section{Introduction}

We consider the constrained smooth optimization problem
\begin{equation}\label{eq:intro-prob1}
\min\{\, f(p):\ p\in C\,\},
\end{equation}
where $\mathbb M$ is a Hadamard manifold, $C\subset \mathbb M$ is a closed geodesically convex set, and $f:\mathbb M\to\mathbb R$ is continuously differentiable. When $f$ is geodesically convex on $C$, problem~\eqref{eq:intro-prob1} falls within the scope of convex optimization in nonpositively curved spaces \cite{Udriste1994,Bacak2014}. Optimization problems of the form \eqref{eq:intro-prob1} have recently attracted considerable attention and have been studied using a variety of first-order algorithms; see, for example, \cite{WeberSra2022,louzeiro2022projected,bergmann2025projection,ren2026riemannian,Sra2016first,deng2025decentralized}.

First-order methods on Riemannian manifolds are now well established for unconstrained optimization; see, for example, \cite{Absil2008,Boumal2020} and the references therein. The constrained setting is substantially more challenging because feasibility must be preserved throughout the iterations, while the behavior of projection-type operators depends strongly on the underlying geometry. A fundamental advantage of Hadamard manifolds is that the exponential map is globally defined and every pair of points is connected by a unique minimizing geodesic. Consequently, logarithm maps and geodesic segments are available globally. Moreover, if $C$ is closed and geodesically convex, then the metric projection ${\cal P}_C$ is single-valued and satisfies several properties analogous to those of projections in Hilbert spaces, including nonexpansiveness \cite{Walter1974}.

Projected-gradient and closely related proximal-gradient methods have already been studied on Hadamard manifolds. In particular, the fixed-step projected-gradient iteration can be recovered from the intrinsic proximal-gradient framework of \cite{Feng2022} by taking the nonsmooth term as the indicator function of \(C\). Nevertheless, this formal inclusion does not make a direct study of projected-gradient methods redundant. The projection structure is considerably more specific and allows one to exploit geometric properties of the metric projection, to work with a natural projection-based stationarity measure, and to derive convergence results designed to constrained smooth optimization. Moreover, the available proximal-gradient results do not directly cover the setting considered here: a backtracking intrinsic proximal-gradient method without assuming Lipschitz continuity of the Riemannian gradient was recently studied in \cite{BentoSantiago2026}, but under a coercivity assumption, whereas the closest direct predecessor of our backtracking projected-gradient method is \cite{bergmann2025projection}, where a projected gradient step followed by an Armijo line search along the feasible geodesic segment is analyzed for possibly nonconvex objectives, but only on hyperbolic space forms. Thus, despite the fact that projected gradient can be embedded into a more general proximal-gradient framework, its direct analysis on general Hadamard manifolds still presents a genuine gap. Given the fundamental role of projected-gradient methods in constrained optimization, closing this gap is both natural and relevant. To the best of our knowledge, an analysis on general Hadamard manifolds of the backtracking projected-gradient method considered here, in which the parameter used to compute the projected point and the Armijo steplength along the corresponding feasible geodesic are decoupled, without compactness or coercivity assumptions in the cluster-point stationarity result, is not currently available.

In this paper, we analyze two projected-gradient schemes for solving problem~\eqref{eq:intro-prob1}, namely,  a constant-stepsize scheme and a backtracking scheme. Both algorithms combine a Riemannian gradient step with a metric projection onto the feasible set and then update the iterate along the geodesic segment joining the current point to its projection. The constant-stepsize scheme is analyzed under the assumption that the objective function has a Lipschitz continuous Riemannian gradient, whereas the backtracking variant does not require this assumption for its basic convergence analysis. Our focus is a direct projection-based convergence and complexity analysis on general Hadamard manifolds, with particular emphasis on the backtracking scheme. In the backtracking variant, the line-search procedure guarantees sufficient decrease of the objective function and automatically selects admissible stepsizes without requiring any prior knowledge of Lipschitz constants.

Our analysis establishes the fundamental convergence properties of both algorithms and provides a unified complexity framework. In particular, the main contributions of this paper are summarized as follows.

\begin{itemize}

\item For the constant-stepsize scheme, under Lipschitz continuity of the Riemannian gradient, we provide a direct projection-based analysis yielding a sufficient-descent inequality and stationarity of every accumulation point, without assuming compactness of the feasible set.

\item For the backtracking scheme, in which the parameter defining the projected point and the Armijo steplength along the resulting feasible geodesic are decoupled, we prove well-definedness and show that, provided the trial line-search stepsizes are uniformly bounded away from zero, every accumulation point is first-order stationary. This result requires neither Lipschitz continuity of the Riemannian gradient nor compactness or coercivity. Compactness of the feasible set or of the initial lower level set is needed only when existence of accumulation points is required.

\item Under Lipschitz continuity of the Riemannian gradient, we derive iteration-complexity bounds of order \(O(1/\sqrt{N})\) for the stationarity measure \(d(p_k,p_{k+1})\) in the constant-stepsize scheme and for \(d(p_k,z_k)\) in the backtracking scheme. For the latter, the trial line-search stepsizes are additionally assumed to be uniformly bounded away from zero. We also obtain the corresponding \(\varepsilon\)-complexity estimates and asymptotic regularity of the generated sequences under the respective assumptions.

\item Finally, we illustrate the practical performance of the proposed methods by solving constrained Karcher mean problems on the manifold of symmetric positive definite (SPD) matrices under the constraint
\begin{equation}\label{set:num.ex.in}
\left\{X\in\mathbb{P}^n:\;\alpha I\preceq X\right\},
\end{equation}
which arises naturally in several applications; see, for example, \cite{Afsari2011,bhatia2006riemannian}.

\end{itemize}
{\bf Related work.} Gradient-projection ideas along geodesics go back at least to \cite{Luenberger1972}, and other projected-gradient formulations on Riemannian manifolds have been considered, for example, in \cite{Hauswirth2016,deng2025decentralized}. In the Hadamard setting, Zhang and Sra \cite{Sra2016first} study projected first-order methods for geodesically convex optimization, while Martínez-Rubio et al.~\cite{Martinez023} analyze metric-projected Riemannian gradient descent in a strongly geodesically convex setting. Decentralized online and stochastic projected-gradient variants are considered in \cite{ChenSun2024,LiZhong2026}. These results concern settings or assumptions different from those considered here and therefore do not subsume the present analysis. A closely related line of work arises from proximal-gradient methods. Feng et al.~\cite{Feng2022} develop an intrinsic fixed-step proximal-gradient method for possibly nonconvex problems on Hadamard manifolds; when the nonsmooth term is the indicator function of $C$, the corresponding proximal step reduces to the metric projection onto $C$. Thus, their framework contains the constant-stepsize projected-gradient iteration as a particular case. More recently, Bento and Santiago \cite{BentoSantiago2026} study an intrinsic backtracking proximal-gradient method on Hadamard manifolds without assuming Lipschitz continuity of the gradient, but their convergence analysis relies on a coercivity assumption. The closest work to the present one is \cite{bergmann2025projection}, which directly analyzes constant- and backtracking projected-gradient schemes for possibly nonconvex objectives without compactness or coercivity assumptions, but only on hyperbolic space forms. Consequently, the above results do not cover the projected-direction backtracking scheme studied here on general Hadamard manifolds. In particular, our cluster-point stationarity analysis requires neither Lipschitz continuity of the Riemannian gradient nor compactness or coercivity, while the complexity analysis is obtained under Lipschitz continuity for a natural projection-based stationarity measure. Thus, the present results complement and extend the existing projected- and proximal-gradient literature rather than being subsumed by it.

The remainder of the paper is organized as follows. Section~\ref{sec:Preliminaries} introduces the geometric background and notation used throughout the paper. Section~\ref{sec:Proj} reviews the main properties of the metric projection onto geodesically convex subsets of Hadamard manifolds. Section~\ref{sec:spd} provides a brief overview of the manifold of symmetric positive definite matrices and derives the metric projection onto the set \eqref{set:num.ex.in}. Section~\ref{sec;opt.cond} establishes first-order optimality conditions for problem~\eqref{eq:intro-prob1}. Sections~\ref{sec:Alg1} and~\ref{sec:Alg2} present the projected gradient methods, without and with backtracking, respectively, and analyze their convergence properties. Section~\ref{sec:ComplAnal} is devoted to the iteration-complexity analysis of the proposed methods. Finally, the last two sections report numerical experiments illustrating the practical performance of the algorithms and conclude the paper with some final remarks.

%%%%%%%%%%%%%%%%%%%%%%%%%%%%%%%%%%%%%%%
\section{Preliminaries}\label{sec:Preliminaries}

In this section, we recall basic concepts and results from Riemannian geometry that will be used throughout the paper. Standard references include \cite{Ballmann1985, BridsonHaefliger1999, Sakai1996}. 

Throughout this paper, $\mathbb{M}$ denotes a finite-dimensional Hadamard manifold and  we use \(\mathbb N:=\{0,1,2,\ldots\}\). For each $p\in\mathbb{M}$, we denote by $T_p\mathbb{M}$ the tangent space of $\mathbb{M}$ at $p$, by $T\mathbb{M}$ its tangent bundle. The norm induced by the Riemannian metric $\langle\cdot,\cdot\rangle$ is denoted by $\|\cdot\|$.

Let $f:\mathbb{M}\to\mathbb{R}$ be a differentiable function. The differential of $f$ at $p\in\mathbb{M}$ is the linear mapping
$
Df(p):T_p\mathbb{M}\to\mathbb{R},
$
defined by
\[
Df(p)[v]
=
\lim_{t\to 0}
\frac{f(\exp_p(tv))-f(p)}{t},
\qquad \forall
v\in T_p\mathbb{M}.
\]
The Riemannian gradient of $f$ at $p$, denoted by $\grad f(p)$, is the unique vector in $T_p\mathbb{M}$ satisfying
$
Df(p)[v]
=
\langle \grad f(p),v\rangle$ for all $v\in T_p\mathbb{M}$.

For a piecewise smooth curve $\gamma:[a,b]\to\mathbb{M}$, we denote its length by $\ell(\gamma)$. The Riemannian distance between two points $p,q\in\mathbb{M}$ is denoted by $d(p,q)$. Since $\mathbb{M}$ is complete, the metric space $(\mathbb{M},d)$ is complete and induces the original manifold topology. For each $p\in\mathbb{M}$, the exponential map $\exp_p:T_p\mathbb{M}\to\mathbb{M}$ is given by $\exp_p(v)=\gamma_{p,v}(1)$,
where $\gamma_{p,v}$ is the geodesic satisfying $\gamma_{p,v}(0)=p$ and $\gamma'_{p,v}(0)=v$.
Consequently,
$\gamma_{p,v}(t)=\exp_p(tv)$ for all $t\in\mathbb{R}$.

Since $\mathbb{M}$ is a Hadamard manifold, the exponential map $\exp_p$ is a global diffeomorphism for every $p\in\mathbb{M}$. Its inverse is denoted by
$\log_p:\mathbb{M}\to T_p\mathbb{M}$.
Moreover,
$
d(p,q)=\|\log_p q\|$ for all  $p,q\in\mathbb{M}$. For a fixed $q\in\mathbb{M}$, the distance function
$
d_q:\mathbb{M}\setminus\{q\}\to\mathbb{R}$ defined by
$d_q(p)=d(p,q)$
is of class $C^\infty$, and
$
\grad d_q(p)
=
(-\log_p q)/d(p,q)$ for all $ p\neq q$.
Furthermore, the function
$
d_q^2:\mathbb{M}\to\mathbb{R}$ defined by $
d_q^2(p)=d^2(p,q)$
is $C^\infty$ and satisfies
$
\grad d_q^2(p)=-2\log_p q$ for all $p\in\mathbb{M}$.

Let $\{p_k\}$ and $\{q_k\}$ be sequences in $\mathbb{M}$ converging to $\bar p$ and $\bar q$, respectively. Then
$
\lim_{k\to\infty}\log_{p_k}q_k
=
\log_{\bar p}\bar q$. For $p,q\in\mathbb{M}$, we denote by $\gamma_{pq}:[0,1]\to\mathbb{M}$ the unique geodesic segment joining $p$ and $q$, namely,
$\gamma_{pq}(0)=p$ and $\gamma_{pq}(1)=q$. Its length equals $d(p,q)$. The parallel transport along $\gamma_{pq}$ from $p$ to $q$ is denoted by
$P_{pq}:T_p\mathbb{M}\to T_q\mathbb{M}$.

The following comparison result for geodesic triangles is a direct consequence of \cite[Proposition 4.5]{Sakai1996}.

\begin{lemma}\label{le:CosLawF}
Let $\mathbb{M}$ be a Hadamard manifold. Then
\begin{equation*}
d^2(x,y)+d^2(x,z)
-2\langle\log_x y,\log_x z\rangle
\le d^2(y,z),
\qquad
\forall\,x,y,z\in\mathbb{M}.
\end{equation*}
\end{lemma}

The next definition is standard in optimization on Riemannian manifolds.

\begin{definition}\label{Def:GradLips}
Let $\Omega\subseteq\mathbb{M}$ and let $f:\Omega\to\mathbb{R}$ be differentiable. The gradient vector field $\grad f$ is said to be Lipschitz continuous on $\Omega$ with Lipschitz constant $L\ge0$ if
\[
\|P_{pq}\grad f(p)-\grad f(q)\|
\le
L\,d(p,q),
\qquad
\forall\, p,q\in\Omega.
\]
\end{definition}
A subset $C\subset \mathbb{M}$ is said to be convex if$\gamma_{pq}(t)\in C$ for all$p,q\in C$ and  $t\in[0,1]$.  A function \(f:C\to\mathbb R\) is said to be geodesically convex if \(f(\gamma_{pq}(t))\le (1-t)f(p)+tf(q)\) for all \(p,q\in C\) and \(t\in[0,1]\); it is strictly geodesically convex if the inequality is strict whenever \(p\neq q\) and \(t\in(0,1)\); and it is \(\mu\)-strongly geodesically convex, with \(\mu>0\), if \(f(\gamma_{pq}(t))\le (1-t)f(p)+tf(q)-\frac{\mu}{2}t(1-t)d^2(p,q)\) for all \(p,q\in C\) and \(t\in[0,1]\).

The following lemma provides the Riemannian counterpart of the classical descent estimate for differentiable functions with Lipschitz continuous gradient on convex sets; see \cite[Proposition 10.54]{Boumal2020}.

\begin{lemma}\label{le:lc}
Let $\Omega\subseteq\mathbb{M}$ and let $f:\Omega\to\mathbb{R}$ be differentiable. Assume that $C\subseteq\Omega$ is convex and that $\grad f$ is Lipschitz continuous on $C$ with constant $L\ge0$. Then
\[
f(q)
\le
f(p)
+
\langle \grad f(p),\log_p q\rangle
+
\frac{L}{2}d^2(p,q),
\qquad
\forall\, p,q\in C.
\]
\end{lemma}

%---------------------------------------------------------------------------------------------------------------
\subsection{Metric Projection}\label{sec:Proj}
Let \(C\subseteq\mathbb M\) be a nonempty closed set and let \(q\in\mathbb M\). Consider the constrained optimization problem
\begin{equation}\label{d:copf}
\min_{p\in C} d(p,q).
\end{equation}
Since \(\mathbb M\) is a finite-dimensional complete Riemannian manifold, the Hopf--Rinow theorem implies that closed bounded subsets of \(\mathbb M\) are compact. Any minimizing sequence for problem~\eqref{d:copf} is bounded and therefore admits a convergent subsequence. Because \(C\) is closed, the limit belongs to \(C\), and continuity of the distance function shows that the minimum is attained. Such a minimizer is called a projection of \(q\) onto \(C\) and is denoted by \({\cal P}_C(q)\). It is worth noting that this concept was originally studied in~\cite{Walter1974}.

The next lemma collects some basic properties of the projection mapping 
${\cal P}_C : \mathbb M\to C$.  
For complete proofs, see 
\cite[Thm.~2.1.12]{bacak2014convex} and \cite[Cor.~3.1]{Ferreira2002}.

\begin{lemma}\label{Lem:nonexp.int.neg}
Let $\mathbb M$ be a Hadamard manifold, and let $C \subseteq \mathbb M$ be a closed convex set.  
Then the projection ${\cal P}_C : \mathbb M\to C$ has the following properties:
\begin{enumerate}
\item[(i)] it is single-valued;

\item[(ii)] it is nonexpansive, i.e., the inequality
$d\!\left({\cal P}_C(p),\, {\cal P}_C(q)\right) \leq d(p, q)$ holds for all $p, q \in \mathbb M$. Consequently, ${\cal P}_C$ is continuous;

\item[(iii)] for every $q \in \mathbb M$, the inequality
$
\langle 
\log_{{\cal P}_C(q)} q,\,
\log_{{\cal P}_C(q)} p
\rangle 
\leq 0
$ holds for  all $p \in C$.
\end{enumerate}
\end{lemma}

The next lemma derives a fundamental inequality for metric projections onto closed convex sets in Hadamard manifolds. In particular, it yields an estimate for projected gradient-type steps that plays a central role in the analysis of the proposed algorithm.

\begin{lemma} \label{le:coslawap}
Let $C \subseteq \mathbb M$ be a closed convex set and let $q \in \mathbb M$. Then,
\begin{equation}\label{eq:dlcf}
d^2(p, {\cal P}_C(q)) \leq \left\langle \log_p q, \log_p {\cal P}_C(q) \right\rangle,
\quad \forall\, p \in C.
\end{equation}
In particular, for every $\alpha > 0$ and every $v \in T_p\mathbb M$, 
\begin{equation} \label{eq;snlpf}
\left\langle v, \log_{p} {\cal P}_C\left( \exp_{p}(-\alpha v) \right) \right\rangle 
\leq - \frac{1}{\alpha} d^2\left(p, {\cal P}_{C}(\exp_{p}(-\alpha v))\right).
\end{equation}
\end{lemma}
\begin{proof} 
Applying Lemma~\ref{le:CosLawF} with $ (x,y,z)=\bigl(p,{\cal P}_C(q),q\bigr) $ and $ (x,y,z)=\bigl({\cal P}_C(q),p,q\bigr)$,  we obtain 
\begin{align*}
 d^2\bigl(p,{\cal P}_C(q)\bigr) +d^2(p,q) -2\Big\langle \log_p{\cal P}_C(q), \log_p q \Big\rangle &\le d^2\bigl({\cal P}_C(q),q\bigr)\\
  d^2\bigl({\cal P}_C(q),p\bigr) +d^2\bigl({\cal P}_C(q),q\bigr) -2\Big\langle \log_{{\cal P}_C(q)}p, \log_{{\cal P}_C(q)}q \Big\rangle &\le d^2(p,q). 
  \end{align*}
  Combining the above inequalities and using Lemma~\ref{Lem:nonexp.int.neg}\,($iii$), we obtain 
  $$
   d^2\bigl(p,{\cal P}_C(q)\bigr) \le \left\langle \log_p q, \log_p {\cal P}_C(q) \right\rangle + \left\langle \log_{{\cal P}_C(q)} q, \log_{{\cal P}_C(q)} p \right\rangle \le \left\langle \log_p q, \log_p {\cal P}_C(q) \right\rangle,
   $$
   which proves \eqref{eq:dlcf}. Now let $\alpha>0$ and $v\in T_p\mathbb M$. Applying \eqref{eq:dlcf} with $ q=\exp_p(-\alpha v)$,  and using the identity $ \log_p\bigl(\exp_p(-\alpha v)\bigr)=-\alpha v$,  we obtain $ d^2\bigl(p,{\cal P}_C(\exp_p(-\alpha v))\bigr) \le \left\langle -\alpha v, \log_p{\cal P}_C(\exp_p(-\alpha v)) \right\rangle$.  Dividing both sides by $\alpha>0$ yields \eqref{eq;snlpf}. 
   \end{proof}

%---------------------------------------------------------------------------------------------------------------
\subsection{SPD manifold}\label{sec:spd}
Let \( \mathbb{R}^{n \times n} \) denote the set of real matrices of order \( n \times n \),  
\( \mathbb{S}^n\subset \mathbb{R}^{n \times n} \) the set of symmetric matrices, and  
\( \mathbb{P}^{n}\subset \mathbb{R}^{n \times n} \) the cone of symmetric positive definite (SPD)  matrices.  
Define the inner product
\begin{equation*} \label{eq:metric}
\langle U, V \rangle_{X} := \operatorname{tr}\!\left(V X^{-1} U X^{-1}\right), 
\quad X \in \mathbb{P}^n, \; U, V \in \mathbb{S}^n,
\end{equation*}
where \( \operatorname{tr}(\cdot) \) denotes the trace operator.  It is well known that  
\( \mathbb{M} = (\mathbb{P}^n, \langle \cdot, \cdot \rangle) \)  
is a Hadamard manifold (see, for instance,~\cite[Theorem~1.2, p.~325]{Lang1999}),  
and that the tangent space \( T_X \mathbb{M} \) can be naturally identified with \( \mathbb{S}^n\) for every \( X \in \mathbb{M} \). For simplicity, we will henceforth denote 
$(\mathbb{P}^n, \langle \cdot, \cdot \rangle)$ by $\mathbb{P}^n$.
Moreover, its sectional curvature is bounded below by \( -1/2 \); see~\cite[Proposition~A.2]{louzeiro2022projected}.
For \( X \in \mathbb{P}^n\) and \( V \in T_X \mathbb{P}^n\),  
the exponential and logarithmic maps are given respectively by
\begin{equation}\label{eq:exp.log.matrix}
\exp_X(V) = X^{1/2}\Exp\!\big(X^{-1/2} V X^{-1/2}\big)X^{1/2}, 
\qquad
\log_X(Y) = X^{1/2}\Log\!\big(X^{-1/2} Y X^{-1/2}\big)X^{1/2},
\end{equation}
where \( \Exp(\cdot) \) and \( \Log(\cdot) \) denote the standard matrix exponential and logarithm functions. Consequently, the Riemannian distance between \(X,Y\in\mathbb{P}^n\) is
given by
\begin{equation}\label{eq:spd-distance}
d(X,Y)
=
\left\|
\Log\!\left(X^{-1/2}YX^{-1/2}\right)
\right\|_F
=
\left[ \sum_{i=1}^n\log^2(\mu_i) \right]^{1/2},
\end{equation}
where \(\|\cdot\|_F\) denotes the Frobenius norm, and
\(\mu_1,\ldots,\mu_n\) are the eigenvalues of
\(X^{-1/2}YX^{-1/2}\).
Moreover, for every $X,Y\in \mathbb{P}^n$, the geodesic segment $ \gamma_{X,Y}:[0,1] \to \mathbb{P}^n$ is given by
\begin{equation}\label{eq:geodesic.matrix}
    \gamma_{X,Y}(t)
    :=
    X^{1/2}
    \left(
        X^{-1/2}YX^{-1/2}
    \right)^t
    X^{1/2}, \qquad  t\in [0,1].
\end{equation}
Let \(f:\mathbb{P}^n\to\mathbb{R}\) be twice continuously differentiable. The gradient and Hessian of \(f\) are given by
\begin{align}
\operatorname{grad} f(X) &= X f'(X) X, \label{SPD:Grad}\\
\operatorname{hess} f(X)V
&=
X f''(X)VX
+\frac{1}{2}\Bigl[
V f'(X)X
+
X f'(X)V
\Bigr],
\label{SPD:Hess}
\end{align}
where \(V \in T_X\mathbb{P}^n \), and \(f'(X)\) and \(f''(X)\) are the Euclidean gradient and Hessian of \(f\) at \(X\) with respect to the Frobenius inner product, respectively. 

\begin{definition}[Loewner Partial Order]
Let $A,B \in \mathbb{S}^n$. We say that $A \preceq B$ if $B - A$ is positive semidefinite, and that $A \prec B$ if $B - A$ is positive definite.
\end{definition}

We conclude this section by establishing a basic property of the set employed in the numerical experiments of Section~\ref{sec:nume.exp}.

\begin{lemma}
Let \(\alpha>0\). Then \(C:=\{X\in\mathbb{P}^n:\alpha I\preceq X\}\) is nonempty, closed, and geodesically convex in \(\mathbb{P}^n\).
\end{lemma}
\begin{proof}
Since \(\alpha I\in C\), the set \(C\) is nonempty. We next show that \(C\) is closed. Let \(\{X_k\}\subset C\) be a sequence converging to some \(X\in\mathbb{P}^n\). Since \(X_k\succeq\alpha I\) for every \(k\), we have \(v^\top X_kv\ge\alpha\|v\|^2\) for all \(v\in\mathbb{R}^n\). Passing to the limit as \(k\to\infty\) yields \(v^\top Xv\ge\alpha\|v\|^2\) for all \(v\in\mathbb{R}^n\), which is equivalent to \(X\succeq\alpha I\). Hence \(X\in C\), and thus \(C\) is closed. It remains to prove geodesic convexity. Let \(X,Y\in C\) and \(t\in[0,1]\), and define \(A:=\alpha^{-1}X\) and \(B:=\alpha^{-1}Y\). Then, by the definition of \(C\) and \eqref{eq:geodesic.matrix},
\begin{equation}\label{eq:spd-lem}
A\succeq I,\qquad B\succeq I,\qquad \gamma_{X,Y}(t)=\alpha\,\gamma_{A,B}(t).
\end{equation}
Consequently, \(A^{-1/2}BA^{-1/2}\succeq A^{-1/2}IA^{-1/2}=A^{-1}\). Since the map \(\mathbb{P}^n\ni T\mapsto T^t\), for \(t\in[0,1]\), is operator monotone; see \cite[Theorem~V.1.9]{Bhatia1997}, it follows that \(\left(A^{-1/2}BA^{-1/2}\right)^t\succeq(A^{-1})^t\). Moreover, since \(A\succeq I\), every eigenvalue of \(A^{-1}\) belongs to \((0,1]\), and hence \((A^{-1})^t\succeq A^{-1}\). Therefore,
\[
\left(A^{-1/2}BA^{-1/2}\right)^t\succeq A^{-1}.
\]
Multiplying both sides by \(A^{1/2}\) and using \eqref{eq:geodesic.matrix}, we obtain
\[
I=A^{1/2}A^{-1}A^{1/2}\preceq A^{1/2}\left(A^{-1/2}BA^{-1/2}\right)^tA^{1/2}=\gamma_{A,B}(t).
\]
Finally, \eqref{eq:spd-lem} gives \(\gamma_{X,Y}(t)=\alpha\,\gamma_{A,B}(t)\succeq\alpha I\), and hence \(\gamma_{X,Y}(t)\in C\). Therefore, \(C\) is geodesically convex.
\end{proof}

The preceding lemma ensures that \(C=\{X\in\mathbb P^n:\alpha I\preceq X\}\) is closed and geodesically convex. The next result gives an explicit formula for its metric projection.
\begin{proposition}\label{prop:spd-projection}
Let \(\alpha>0\) and define
$
C:=\{X\in\mathbb P^n:\alpha I\preceq X\}.
$
Let \(Y\in\mathbb P^n\) have the spectral decomposition
$
Y=Q\operatorname{diag}(\lambda_1,\ldots,\lambda_n)Q^\top,
$
where \(Q\) is orthogonal, \(\lambda_1,\ldots,\lambda_n\) are the
eigenvalues of \(Y\), ordered as
\(\lambda_1\geq\cdots\geq\lambda_n>0\), and
\(\operatorname{diag}(\lambda_1,\ldots,\lambda_n)\) denotes the diagonal
matrix with diagonal entries \(\lambda_1,\ldots,\lambda_n\). Then
\[
\mathcal P_C(Y)
=
Q\operatorname{diag}\bigl(
\max\{\lambda_1,\alpha\},\ldots,
\max\{\lambda_n,\alpha\}
\bigr)Q^\top.
\]
\end{proposition}

\begin{proof}
Fix \(X\in C\) and define
$
\widehat X:=
Q\operatorname{diag}\bigl(
\max\{\lambda_1,\alpha\},\ldots,
\max\{\lambda_n,\alpha\}
\bigr)Q^\top.
$
The definition of \(C\), together with congruence by \(Y^{-1/2}\), yields
$
Y^{-1/2}XY^{-1/2}\succeq\alpha Y^{-1}.
$
Let \(\mu_1\geq\cdots\geq\mu_n>0\) denote the eigenvalues of
\(Y^{-1/2}XY^{-1/2}\). Since the eigenvalues of \(\alpha Y^{-1}\),
arranged in nonincreasing order, are
\(\alpha/\lambda_n\geq\cdots\geq\alpha/\lambda_1\), the monotonicity of
the ordered eigenvalues with respect to the Loewner order gives
\[
\mu_i\geq\frac{\alpha}{\lambda_{n+1-i}},
\qquad i=1,\ldots,n.
\]
Let \(r\) be the number of eigenvalues of \(Y\) that are strictly smaller
than \(\alpha\). Then
\[
\mu_i\geq\frac{\alpha}{\lambda_{n+1-i}}>1,
\qquad i=1,\ldots,r.
\]
Therefore, using the expression for the affine-invariant Riemannian distance,
\[
\begin{aligned}
d^2(Y,X)
&= \sum_{i=1}^n \log^2(\mu_i)
 \geq \sum_{i=1}^r \log^2(\mu_i)
 \geq \sum_{i=1}^r
 \log^2\left(\frac{\alpha}{\lambda_{n+1-i}}\right)\\
&= \sum_{\lambda_j<\alpha}
\log^2\left(\frac{\alpha}{\lambda_j}\right)
= d^2(Y,\widehat X).
\end{aligned}
\]
Since \(\widehat X\in C\), we conclude that
$
\widehat X=\mathcal P_C(Y),
$
which completes the proof.
\end{proof}

%---------------------------------------------------------------------------------------------------------------
\section{Constrained Optimization on Hadamard Manifolds}\label{sec;opt.cond}

Consider a constrained optimization problem on a Hadamard manifold \( \mathbb M\) of the form
\begin{equation} \label{eq:OptP}
    \min \{ f(p) : p \in C \},
\end{equation}
where \( C \subseteq \mathbb M\) is a closed convex set and \( f : \mathbb M\to \mathbb{R} \) is a continuously differentiable function. Let \( \Omega^* \neq \varnothing \) denote the solution set of  \eqref{eq:OptP}, and let 
\( f^* := \inf_{p \in C} f(p) > -\infty \) be the optimal value of \( f \).

\begin{definition}\label{def:sta.point}
A point \( \bar{p} \in C \) is called a stationary point of Problem 
\eqref{eq:OptP} if it satisfies
$$
\left\langle \grad f(\bar{p}), \log_{\bar{p}} p \right\rangle \ge 0,
\qquad \forall\, p \in C .
$$
\end{definition}

The next result shows that every  solution of \eqref{eq:OptP} is a stationary point. Consequently, stationarity constitutes a necessary first-order optimality condition for ~\eqref{eq:OptP}.

 \begin{proposition}
Let $\mathbb M$ be a Hadamard manifold, and let $C \subseteq \mathbb M$ be a closed convex set.  
If $\bar{p} \in C$ is a solution to Problem~\eqref{eq:OptP}, then $\bar{p}$ is a stationary point.
\end{proposition}
\begin{proof}
Suppose, by contradiction, that \( \bar{p} \) is not a stationary point of
\eqref{eq:OptP}. Then there exists \( \tilde{p} \in C \) such that
\begin{equation} \label{Vphi<0}
\left\langle \grad f(\bar{p}), \log_{\bar{p}} \tilde{p} \right\rangle < 0 .
\end{equation}
Define \( \varphi : \mathbb{R} \to \mathbb{R} \) by
$\varphi(t) := f\!\left(\exp_{\bar{p}}\!\big(t \log_{\bar{p}}\tilde{p}\big)\right)$ for all $t \in \mathbb{R}$.
Since both \( f \) and \( \exp_{\bar{p}} \) are continuously differentiable, it follows
that \( \varphi \) is also continuously differentiable.  
Applying the chain rule and using~\eqref{Vphi<0}, we obtain
$
\varphi'(0)
= \left\langle \grad f(\bar{p}), \log_{\bar{p}} \tilde{p} \right\rangle
< 0 .
$
Therefore, there exists \( \varepsilon \in (0,1) \) such that  
\begin{equation}\label{dVphi}
\varphi'(t) < 0 , \qquad \forall\, t \in [0,\varepsilon].
\end{equation}
From the definition of \( \varphi \) and from \eqref{dVphi}, it follows that
\begin{equation}\label{ineq:contr}
f\!\left(\exp_{\bar{p}}\!\big(\varepsilon \log_{\bar{p}}\tilde{p}\big)\right)
- f(\bar{p}) =\varphi(\varepsilon) - \varphi(0)
= \int_0^{\varepsilon} \varphi'(t)\, dt
< 0 .
\end{equation}
On the other hand, since \( C \) is convex and both \( \bar{p} \) and  
\( \tilde{p} \in C \), the point  
$
\exp_{\bar{p}}\!\big(\varepsilon \log_{\bar{p}}\tilde{p}\big)$,
which lies on the geodesic segment connecting \( \bar{p} \) to \( \tilde{p} \), also
belongs to \( C \).  
Hence, inequality~\eqref{ineq:contr} contradicts the fact that \( \bar{p} \) is a
solution of \eqref{eq:OptP}.  
Therefore, \( \bar{p} \) must be a stationary point.
\end{proof}
The next proposition provides conditions involving the projection mapping \( \mathcal{P}_C \) to determine whether a point is stationary for Problem~\eqref{eq:OptP}.

\begin{proposition} \label{pr:ProjProperty}
Let $\mathbb M$ be a Hadamard manifold, and let $C \subseteq \mathbb M$ be a closed convex set.  
Let \( {\bar p} \in C \) be such that \( \grad f(\bar p) \neq 0 \) and \( \alpha > 0 \). Then the following statements hold:
\begin{enumerate}
\item[(i)] \( {\bar p} \) is a stationary point  if and only if  
 ${\bar p} = {\cal P}_C\left(\exp_{{\bar p}}(-\alpha\grad f({\bar p}))\right)$.
\item[(ii)] If \( {p} \) is not a stationary point, then
\[
\langle \grad f(p), \log_{p} {\cal P}_{C}(\exp_{p}(-\alpha \grad f(p))) \rangle < 0.
\]
In particular, if there exists \( {\bar \alpha} > 0 \) such that
\[
\langle \grad f(\bar p), \log_{\bar p} {\cal P}_{C}(\exp_{\bar p}(-{\bar \alpha} \grad f(\bar p))) \rangle \geq 0,
\]
then \( \bar{p} \) is a stationary point.
\end{enumerate}
\end{proposition}

\begin{proof}
To prove item~$(i)$, first assume that \( \bar p \in C \) is a stationary point of problem~\eqref{eq:OptP}; that is,
\[
\langle \grad f(\bar p), \log_{\bar p} p \rangle \geq 0, \quad \forall p \in C.
\]
Applying this inequality to 
$
p = {\cal P}_C\!\left(\exp_{\bar p}(-\alpha \grad f(\bar p))\right)
$
and using Lemma~\ref{le:coslawap} with \( v = \grad f(\bar p) \) and $p=\bar{p}$, we obtain
\[
0 
\le 
\left\langle \grad f(\bar p), 
\log_{\bar p} {\cal P}_C\!\left(\exp_{\bar p}(-\alpha \grad f(\bar p))\right) \right\rangle
\le 
-\frac{1}{\alpha}\,
d^2\!\left(\bar p, {\cal P}_C\!\left(\exp_{\bar p}(-\alpha \grad f(\bar p))\right)\right)
\le 0.
\]
Hence,
$
d\!\left(\bar p, {\cal P}_C\!\left(\exp_{\bar p}(-\alpha \grad f(\bar p))\right)\right) = 0,
$
which implies that
$
\bar p = {\cal P}_C\!\left(\exp_{\bar p}(-\alpha \grad f(\bar p))\right).
$
Conversely, suppose that 
$
{\bar p} = {\cal P}_C(\exp_{\bar p}(-\alpha \grad f(\bar p))).
$
Using this and Lemma~\ref{Lem:nonexp.int.neg} $(iii)$ with \( q = \exp_{\bar p}(-\alpha \grad f(\bar p)) \), we obtain
\[
\langle -\alpha \grad f(\bar p), \log_{\bar p} p \rangle  
= \langle \log_{\bar p}(\exp_{\bar p}(-\alpha \grad f(\bar p))), \log_{\bar p} p \rangle 
\leq 0, \quad \forall p \in C.
\]
Dividing by \( -\alpha \) yields  
$\langle \grad f(\bar p), \log_{\bar p} p \rangle \ge 0$ for all $p \in C$,
and thus \( \bar p \) is stationary. For item~$(ii)$, let \( p \) be a non-stationary point.  
By item~$(i)$, this means
$
p \neq {\cal P}_C(\exp_p(-\alpha \grad f(p))).
$
Using Lemma~\ref{le:coslawap} with $v=\grad f(p)$, it follows that 
\[
0 < \frac{1}{\alpha} d^2\left(p, {\cal P}_C\left(\exp_p(-\alpha \grad f(p))\right)\right) 
\leq - \langle \grad f(p), \log_{p} {\cal P}_C\left(\exp_p(-\alpha \grad f(p))\right) \rangle,
\]
which proves the first part of item~$(ii)$.  
Finally, the second part of item~$(ii)$ follows directly from the contrapositive of the first part.
\end{proof}

%---------------------------------------------------------------------------------------------------------------
\section{Projected Gradient  Method on Hadamard Manifolds} \label{sec:socop}
In this section, we propose a    projected gradient method on Hadamard manifolds (PGMHM) for solving problem~\eqref{eq:OptP}. The method combines Riemannian gradient steps with metric projections onto the feasible set \(C\). We consider both constant and backtracking step-size strategies. we show that every accumulation point of the generated sequences is stationary and derive iteration-complexity bounds.

%---------------------------------------------------------------------------------------------------------------
\subsection{Projected Gradient  Method with Constant Step Size}\label{sec:Alg1}

We first consider the PGMHM with a constant step-size. Under a Lipschitz continuity assumption on the gradient vector field, the method admits a sufficient descent property that yields monotonic decrease of the objective function and provides a natural measure of first-order stationarity. As a consequence, either the algorithm terminates at a stationary point in finitely many iterations or every accumulation point of the generated sequence is stationary. 

Throughout this subsection, the following regularity condition will be needed:
\begin{itemize}
\item[{\bf (A1)}] The gradient vector field \( \grad f \) is Lipschitz continuous on \( C \) with constant \( L >0 \).
\end{itemize}

The projected gradient method with constant step size for solving problem~\eqref{eq:OptP} is stated as follows.

\begin{algorithm}[H]
\caption{Riemannian Projected Gradient Method with Constant Step Size}\label{Alg1s}
	\begin{algorithmic}[1]
		\State Choose \(0 < \alpha < 2/L\) and an initial point \( p_0 \in C \). Set \( k = 0 \).
		\State  Compute
		\begin{equation}\label{step.arm}
		p_{k+1} := {\cal P}_C(y_k)
		\qquad	y_k := \exp_{p_k}\!\left(- \alpha \grad f(p_k)\right),	
		\end{equation}
		\State  If \( p_{k+1}=p_k\), then \textbf{stop} and return \( p_{k+1} \). Otherwise, set \( k \leftarrow k + 1 \) and go to Step~\textbf{2}.
	\end{algorithmic}
\end{algorithm}

In this subsection, we assume that the algorithm does not stop, that is,
$p_{k+1} \neq p_k$, for all
$k\in\mathbb{N}$.
Consequently, \(\grad f(p_k)\neq 0\) for all \(k\in\mathbb{N}\).

\begin{lemma}\label{Le:MainConvP}
Assume that \textbf{(A1)} holds, and let $\{p_k\} \subseteq C$ be the sequence generated by Algorithm~\ref{Alg1s}. Then
\begin{equation}\label{eq:MainIneqP}
    f(p_{k+1}) 
    \le f(p_k) - \Gamma \, d^2(p_k, p_{k+1}),
    \qquad 
    \Gamma := \frac{2 - \alpha L}{2\alpha} > 0,
\end{equation}
for all $k \in \mathbb{N}$. In particular, the sequence $\{ f(p_k)\}$ is nonincreasing and convergent.
\end{lemma}

\begin{proof}
Fix any \( k \in \mathbb{N} \). Since \( f \) has an \( L \)-Lipschitz continuous gradient, Lemma~\ref{le:lc}, applied with \( p = p_k \) and \( q = p_{k+1} \), together with the definition of \( y_k \) in \eqref{step.arm}, yields
\begin{align*}
f(p_{k+1})
&\le f(p_k) 
    + \left\langle \grad f(p_k), \log_{p_k} p_{k+1} \right\rangle
    + \frac{L}{2} \, d^{2}(p_k, p_{k+1}) \\
&= f(p_k) 
    - \frac{1}{\alpha} 
      \left\langle \log_{p_k} y_k, \log_{p_k} p_{k+1} \right\rangle
    + \frac{L}{2} \, d^{2}(p_k, p_{k+1}).
\end{align*}
On the other hand, using the first equality in \eqref{step.arm} and applying Lemma~\ref{le:coslawap} with  
\( p = p_k \) and \( q = y_k \), we obtain
\begin{equation*}
-\left\langle \log_{p_k} y_k, \log_{p_k} p_{k+1} \right\rangle
= 
-\left\langle \log_{p_k} y_k, \log_{p_k} {\cal P}_C(y_k) \right\rangle
\le -\, d^{2}(p_k, {\cal P}_C(y_k))
= -\, d^{2}(p_k, p_{k+1}).
\end{equation*}
Therefore, combining the inequalities above and using the definition of 
\( \Gamma \) in \eqref{eq:MainIneqP}, we obtain
\begin{align*}
f(p_{k+1}) 
&\le f(p_k) 
    - \frac{1}{\alpha} d^2(p_k, p_{k+1}) 
    + \frac{L}{2} d^2(p_k, p_{k+1}), \\
&= f(p_k) 
    - \left( \frac{2 - \alpha L}{2\alpha} \right)
      d^2(p_k, p_{k+1}) , \\
      &= f(p_k) 
    - \Gamma
      d^2(p_k, p_{k+1}).
\end{align*}
Since \(0 < \alpha < 2/L\), we have \( \Gamma > 0 \), and therefore the sequence 
\( \{f(p_k) \} \) is nonincreasing.
Furthermore, because \( f^* > -\infty \), it follows that \( \{f(p_k)\} \) converges, 
which completes the proof.
\end{proof}

We now show that every accumulation point of the sequence generated by Algorithm~\ref{Alg1s} is stationary for \eqref{eq:OptP}.

\begin{theorem}\label{teo.Main}
Assume that~\textbf{(A1)} holds, and let \(\{p_k\}\subseteq C\) be the sequence generated by Algorithm~\ref{Alg1s}. Then every accumulation point \(\bar p\) of \(\{p_k\}\) is a stationary point of Problem~\eqref{eq:OptP}.
\end{theorem}
\begin{proof}
If \( \grad f(\bar p) = 0 \), then the stationarity condition, see Definition\,\ref{def:sta.point}, is trivially satisfied, and therefore \( \bar p \) is a stationary point. 
Assume now that \( \grad f(\bar p) \neq 0 \). 
Let \( \{ p_{k_j}\} \) be a subsequence such that \( p_{k_j} \to \bar p \). 
From Lemma~\ref{Le:MainConvP}, we have
\begin{equation}\label{eq:eK1}
d^{2}(p_{k_j}, p_{k_j+1})
\le \frac{1}{\Gamma}\bigl(f(p_{k_j}) - f(p_{k_j+1})\bigr),
\qquad \forall j \in \mathbb{N}.
\end{equation}
Since Lemma~\ref{Le:MainConvP} also guarantees that the sequence \( \{f(p_k)\} \) converges, the right-hand side of \eqref{eq:eK1} tends to zero as \( j \to + \infty \). 
Consequently,
$\lim_{j \to \infty} d(p_{k_j}, p_{k_j+1}) = 0.$
Because \( p_{k_j} \to \bar p \), it follows that \( p_{k_j+1} \to \bar p \) as well. Next, using \eqref{step.arm} together with the continuity of the exponential map, of the gradient field, and of the metric projection onto \( C \) (see Lemma~\ref{Lem:nonexp.int.neg} $(ii)$), we obtain
\[
\bar p 
= \lim_{j \to \infty} p_{k_j+1}
= \lim_{j \to \infty} P_C\!\left(\exp_{p_{k_j}}(-\alpha \grad f(p_{k_j}))\right)
= {\cal P}_C\!\left(\exp_{\bar p}(-\alpha\, \grad f(\bar p))\right).
\]
By item~$(i)$ of Proposition~\ref{pr:ProjProperty}, this identity implies that \( \bar p \) is a stationary point of \eqref{eq:OptP}, which completes the proof.
\end{proof}

Item~$(i)$ of Proposition~\ref{pr:ProjProperty} states that if \(p_k={\cal P}_C\!\left(\exp_{p_k}(-\alpha\grad f(p_k))\right)\), then \(p_k\) is a stationary point of~\eqref{eq:OptP}. Consequently, the distance between \(p_k\) and \({\cal P}_C\!\left(\exp_{p_k}(-\alpha\grad f(p_k))\right)\), namely \(d(p_k,p_{k+1})\), may be interpreted as a measure of the stationarity of \(p_k\). The next theorem provides an iteration-complexity bound for this quantity.

\begin{theorem}\label{Teo:FCompP}
Assume that \textbf{(A1)} holds, and let  
\(\{p_k\} \subseteq C\) be the sequence generated by  
Algorithm~\ref{Alg1s}.  
Then, for every \(N \in \mathbb{N}\), the following inequality holds:
\[
\min\left\{ d(p_k, p_{k+1}) : k = 0, 1, \ldots, N \right\}
\leq \frac{\sqrt{f(p_0) - f^*}}{\sqrt{\Gamma (N+1)}}.
\]
\end{theorem}

\begin{proof}
Fix \(N \in \mathbb{N}\). Rearranging the terms in inequality~\eqref{eq:MainIneqP}, we obtain
\[
d^2(p_k, p_{k+1}) 
\leq \frac{1}{\Gamma}\big( f(p_k) - f(p_{k+1}) \big),
\qquad \forall\, k = 0,1,\ldots.
\]
Summing both sides from \(k = 0\) to \(k = N\), and recalling that  
\(f^* > -\infty\) denotes the optimal value of Problem~\eqref{eq:OptP}, we find
\[
\sum_{k=0}^{N} d^2(p_k, p_{k+1})
\leq \frac{1}{\Gamma} \sum_{k=0}^{N} \big( f(p_k) - f(p_{k+1}) \big)
= \frac{1}{\Gamma}\big( f(p_0) - f(p_{N+1}) \big)
\leq \frac{1}{\Gamma}\big( f(p_0) - f^* \big).
\]
Therefore,
\[
(N+1)\,
\min\left\{ d^2(p_k, p_{k+1}) : k = 0,\ldots,N \right\}
\leq 
\sum_{k=0}^{N} d^2(p_k, p_{k+1})
\leq 
\frac{1}{\Gamma}\big( f(p_0) - f^* \big),
\]
which immediately yields the claimed bound.
\end{proof}
The complexity estimate of Theorem~\ref{Teo:FCompP} also yields the following asymptotic regularity property of the sequence generated by Algorithm~\ref{Alg1s}.
\begin{corollary}
Under the assumptions of Theorem~\ref{Teo:FCompP}, the successive displacements generated by Algorithm~\ref{Alg1s} satisfy
\[
\sum_{k=0}^{+\infty} d^2(p_k,p_{k+1})\le \frac{f(p_0)-f^*}{\Gamma}<+\infty.
\]
Consequently, \(d(p_{k+1},p_k)\to0\) as \(k\to +\infty\).
\end{corollary}
\begin{proof}
Letting \(N\to +\infty\) in the summability estimate established in the proof of Theorem~\ref{Teo:FCompP} gives the first assertion, and the second follows immediately.
\end{proof}

%---------------------------------------------------------------------------------------------------------------
\subsection{Projected Gradient  Method with Backtracking Step Size}\label{sec:Alg2}

Algorithm~\ref{Alg1s} requires knowledge of a Lipschitz constant for the gradient vector field in order to determine the step size. In practice, however, such a constant may be unavailable or difficult to estimate. To overcome this limitation, we consider in this section a variant of Algorithm~\ref{Alg1s} in which the step size is selected by means of a backtracking procedure. The resulting PGMHM for solving Problem~\eqref{eq:OptP} is stated below.

\begin{algorithm}[H]
\caption{Riemannian Projected Gradient  Method with Backtracking}\label{AlgsIP}
\begin{algorithmic}[1]
\State Choose constants $\rho,\beta\in(0,1)$ and $t_{\min}, t_{\max}>0$.
Select an initial point $p_0\in C$ and set $k\gets 0$.
\State Choose a step-size $t_k$ such that
\begin{equation}\label{eq:coaIP}
0<t_{\min}\le t_k\le  t_{\max} .
\end{equation}
\State Compute
\begin{equation}\label{eq:cpIP}
y_k:=\exp_{p_k}\!\bigl(-t_k\,\grad f(p_k)\bigr),\qquad
z_k:={\cal P}_C(y_k).
\end{equation}
\State   If $z_k=p_k$, then \textbf{stop} and return \( p_{k} \). Otherwise, choose any trial step-size $\bar\lambda_k\in(0,1]$ and compute
\begin{equation}\label{eq:jkIP}
\ell_k:=\min\Bigl\{\ell\in\mathbb N:\ 
f\bigl(\exp_{p_k}(\beta^\ell\bar\lambda_k\,v_k)\bigr)\le
f(p_k)+\rho\,\beta^\ell\bar\lambda_k\,\langle \grad f(p_k),v_k\rangle
\Bigr\}, 
\end{equation}
where  $v_k:=\log_{p_k}z_k$.
\State Set 
\begin{equation}\label{eq:pk+1thet}
\lambda_k:=\beta^{\ell_k}\bar\lambda_k, \quad p_{k+1}:=\exp_{p_k}(\lambda_k\,v_k),
\end{equation}
and $k\gets k+1$, and return to Step~2.
\end{algorithmic}
\end{algorithm}

After the stopping criterion is checked, Algorithm~\ref{AlgsIP} computes an intrinsic gradient step from
\(p_k\) by moving along the geodesic generated by the exponential map, yielding the unconstrained trial point
$
y_k=\exp_{p_k}\bigl(-t_k\,\grad f(p_k)\bigr).
$
This point is subsequently projected onto the feasible set,
$
z_k={\cal P}_C(y_k)\in C.
$
The vector
$
v_k=\log_{p_k}z_k
$
represents the initial velocity of the unique minimizing geodesic joining \(p_k\) and \(z_k\), and may therefore be interpreted as a projected descent direction.
The next iterate is obtained through a backtracking line-search procedure along the feasible geodesic segment
$
s\mapsto \exp_{p_k}(s\,v_k).
$
Starting from an initial trial steplength \(\bar{\lambda}_k\in(0,1]\), the algorithm successively reduces the steplength by a factor \(\beta\in(0,1)\) until the Armijo condition~\eqref{eq:jkIP} is satisfied. The accepted steplength is given by
$
\lambda_k=\beta^{\ell_k}\bar{\lambda}_k\in(0,1],
$
and the new iterate is defined by
$
p_{k+1}=\exp_{p_k}(\lambda_k v_k).
$
Since \(z_k\in C\) and \(v_k=\log_{p_k}z_k\), the curve
$
\gamma_k(s):=\exp_{p_k}(s\,v_k)$, $s\in[0,1]$,
is the  geodesic connecting \(p_k\) to \(z_k\). As \(C\) is convex, then
$
\gamma_k([0,1])\subset C,
$
which implies that every trial point of the form
$
\exp_{p_k}(\beta^\ell \bar{\lambda}_k v_k)
$
remains feasible. Consequently, the update rule~\eqref{eq:pk+1thet} preserves feasibility, and the entire sequence \(\{p_k\}\) generated by Algorithm~\ref{AlgsIP} remains in \(C\).

For the convergence analysis developed in this section, we impose the following assumption:

\begin{itemize}
\item[{\bf (A2)}]
$
\inf_{k\in\mathbb N}\bar\lambda_k
=: \lambda^{\star}
\in (0,1].
$
\end{itemize}

Assumption {\bf (A2)} can be readily enforced in practice. Specifically, given any constant $\underline{\lambda} \in(0,1]$, one may restrict the trial step-sizes to satisfy
$\bar\lambda_k\in[\underline{\lambda},1]$ for all $ k\in\mathbb N$.
Consequently,
$
\inf_{k\in\mathbb N}\bar\lambda_k
\geq
\underline{\lambda}>0,
$
and therefore  {\bf (A2)} is satisfied with
$\lambda^{\star}=\underline{\lambda}$. This approach will be employed in some of the trial step-size selection rules introduced in the next section.

%%%%%%%%%%%%%%%%%%%%%%%%%%%%%%%%%%
\subsubsection{Choice of the trial step-size}\label{sec:lk}
%%%%%%%%%%%%%%%%%%%%%%%%%%%%%%%%%%

We propose four strategies for choosing the trial step-size $\bar{\lambda}_k$ in the backtracking rule~\eqref{eq:jkIP} of Algorithm~\ref{AlgsIP} and show that each of them satisfies Assumption~{\bf (A2)}.

\begin{enumerate}[label={(\textbf{St\arabic*})}, ref={(St\arabic*)}]
\item \label{it:ArmijoIP}
({\it Armijo trial step-size})
Set ${\bar \lambda}_k:=1$ for all $k\in\mathbb N$ and take $\beta\in(0,1)$.
Then \eqref{eq:jkIP} reduces to the classical Armijo backtracking rule along the feasible geodesic
$s\mapsto \exp_{p_k}(s\,v_k)$:
\begin{equation*}\label{eq:TkArmlkIP}
\ell_k:=\min\Bigl\{\ell\in\mathbb N:\ 
f\bigl(\exp_{p_k}(\beta^\ell v_k)\bigr)\le f(p_k)+\rho\,\beta^\ell\langle \grad f(p_k),v_k\rangle
\Bigr\},
\end{equation*}
and the accepted step-size is $\lambda_k:=\beta^{\ell_k}$. Since $\bar\lambda_k=1$ for every $k\in\mathbb N$, we have
$\inf_{k\in\mathbb N}\bar\lambda_k=1$, and hence  {\bf (A2)} holds automatically.

\item \label{it:AdaptLIP}
({\it Adaptive trial step-size via a geometric Lipschitz estimate})
Assume that Assumption~\textbf{(A1)} holds, i.e.,   $\grad f$ is $L$-Lipschitz on  $C$.
Fix $\rho,\beta\in(0,1)$ as in Algorithm~\ref{AlgsIP} and choose an initial estimate $L_0\ge 1$; set $L_{-1}:=L_0$.
At iteration $k$, choose the trial step size in \eqref{eq:jkIP} as
$
\bar\lambda_k:= 1/L_{k-1} \in(0,1],
$
compute $\ell_k$ by \eqref{eq:jkIP}, set the accepted step-size $\lambda_k:=\beta^{\ell_k}\bar\lambda_k$,
and update $p_{k+1}:=\exp_{p_k}(\lambda_k v_k)$.
Finally, update the estimate by
$
L_k:=\beta^{-\ell_k}L_{k-1}.
$
Note that
\begin{equation*} \label{eq:defclip}
\bar\lambda_{k+1}:=\frac{1}{L_k}=\frac{1}{\beta^{-\ell_k}L_{k-1}}=\beta^{\ell_k}\frac{1}{L_{k-1}}=\beta^{\ell_k}\bar\lambda_k=\lambda_k.
\end{equation*} 
that is, the next trial step-size is exactly the previously accepted one. Next we will compute a lower bound for the trials. For that, fix $k\in\mathbb N$ and denote by $s:=\beta^\ell\bar\lambda_k\in(0,1]$ a generic candidate tested in
\eqref{eq:jkIP}. By Lemma~\ref{le:lc} along the feasible
geodesic $s\mapsto \exp_{p_k}(s\,v_k)$, we have 
\begin{equation}\label{eq:lbtrials-quad}
f\bigl(\exp_{p_k}(s\,v_k)\bigr)
\le
f(p_k)+s\langle \grad f(p_k),v_k\rangle+\frac{L}{2}s^2\|v_k\|^2,
\qquad s\in[0,1].
\end{equation}
Moreover, by using inequality \eqref{eq;snlpf} in Lemma~\ref{le:coslawap}  applied with $p=p_k$, $v=\grad f(p_k)$ and $\alpha=t_k$, we obtain
\begin{equation}\label{eq:lbtrials-proj}
\|v_k\|^2=d^2(p_k,z_k)\le -\,t_k\,\langle \grad f(p_k),v_k\rangle.
\end{equation}
Substituting \eqref{eq:lbtrials-proj} into \eqref{eq:lbtrials-quad} gives
\begin{equation}\label{eq:lbtrials-combined}
f\bigl(\exp_{p_k}(s\,v_k)\bigr)\leq f(p_k)+ s\left(1-\frac{L}{2}\,s\,t_k\right)\,\langle \grad f(p_k),v_k\rangle, \qquad s\in[0,1].
\end{equation}
Now recall that the Armijo condition in~\eqref{eq:jkIP} at the candidate $s$ is given by 
\begin{equation}\label{eq:lbtrials-armijo}
f\bigl(\exp_{p_k}(s\,v_k)\bigr)
\le
f(p_k)+\rho\,s\,\langle \grad f(p_k),v_k\rangle.
\end{equation}
If $p_k$ is not stationary, then $d(p_k,z_k)>0$ and \eqref{eq:lbtrials-proj} implies
$\langle \grad f(p_k),v_k\rangle<0$. Hence, whenever
\(
1-(L/2)\,s\,t_k\ \ge\ \rho, 
\)
equivalently,  
\begin{equation}\label{eq:lbtrials-small2}
s\ \le\ \frac{2(1-\rho)}{L\,t_k}.
\end{equation}
the right-hand side in \eqref{eq:lbtrials-combined} is no larger than the right-hand side in
\eqref{eq:lbtrials-armijo}, and therefore \eqref{eq:lbtrials-combined} guarantees \eqref{eq:lbtrials-armijo}.
This shows that the backtracking must terminate after finitely many reductions. To extract a uniform lower bound for the trial sequence $\{\bar\lambda_k\}$ under the present update
$\bar\lambda_{k+1}=\lambda_k$, consider two cases:

 {\it Case 1: $\ell_k=0$.}
 
Then $\lambda_k=\bar\lambda_k$ and, by construction, we  have 
$
\bar\lambda_{k+1}=\lambda_k=\bar\lambda_k,
$
so the trial value does not decrease at iteration $k$. 

{\it Case 2: $\ell_k\ge 1$.} 

Set the previous candidate as 
$
s_k^-:=\beta^{\ell_k-1}\bar\lambda_k=\lambda_k/\beta \in(0,1].
$
By minimality of $\ell_k$, the Armijo test fails at $s_k^-$, i.e., the following inequality holds 
\[
f\bigl(\exp_{p_k}(s_k^- v_k)\bigr)
>
f(p_k)+\rho\,s_k^-\,\langle \grad f(p_k),v_k\rangle.
\]
If $s_k^-=\lambda_k/\beta\leq 2(1-\rho)/(L t_k)$, then \eqref{eq:lbtrials-small2} would hold and, as shown above,
\eqref{eq:lbtrials-combined} would guarantee the Armijo inequality at $s_k^-$, a contradiction.
Therefore necessarily
\[
s_k^->\frac{2(1-\rho)}{L\,t_k}\ \ge\ \frac{2(1-\rho)}{L\,t_{\max}}.
\]
Multiplying by $\beta$ and using $\bar\lambda_{k+1}=\lambda_k=\beta s_k^-$ yields
$
\bar\lambda_{k+1}> 2\beta(1-\rho)/(L\,t_{\max}).
$

Combining the two cases and taking the initial value $\bar\lambda_0 = 1/L_0$, it follows that Assumption~{\bf (A2)} holds with
$
\lambda^{\star}
=
\min\!\left\{
1/{L_0},
\,
2\beta(1-\rho)/(t_{\max}\,L)
\right\}.
$

\item \label{it:AdaptIP}
({\it Adaptive trial step-size with a uniform lower bound})
Fix $\beta\in(0,1)$ and choose ${\bar\lambda}_0\in(0,1]$ and $\underline\lambda\in(0,{\bar\lambda}_0]$.
At iteration $k$, take the current trial step-size ${\bar\lambda}_k\in(0,1]$, compute $\ell_k$ by~\eqref{eq:jkIP}, set
$
\lambda_k:=\beta^{\ell_k}\bar\lambda_k\in(0,1]$, and update
$p_{k+1}:=\exp_{p_k}(\lambda_k\,v_k)$.
Then define the next trial step-size by
\begin{equation*}\label{eq:trial-update-IP-lb}
\bar\lambda_{k+1}
:=\max\Bigl\{\underline\lambda,\ \min\{1,\beta^{\ell_k-1}\bar\lambda_k\}\Bigr\}
=\max\Bigl\{\underline\lambda,\ \min\{1,\beta^{-1}\lambda_k\}\Bigr\}.
\end{equation*}
By construction, $\bar\lambda_k \in [\underline{\lambda},1] $, for all $k\in\mathbb{N}$,
which implies that Assumption~{\bf (A2)} holds.
Finally, since $\lambda_k\in(0,1]$ and $z_k={\cal P}_C(y_k)\in C$, if $C$ is convex then the whole
geodesic segment $s\mapsto \exp_{p_k}(s\,v_k)$, $s\in[0,1]$, is contained in $C$. Therefore every trial point
$\exp_{p_k}(\beta^\ell\bar\lambda_k\,v_k)$, as well as the accepted iterate $p_{k+1}$, is feasible.

\item \label{it:St3IP}
({\it Geometric trial-step-size update via a counter, with truncation})
Fix $\beta\in(0,1)$ and choose $\theta_0\in(0,1)$ and $\underline\lambda\in(0,\theta_0]$.
Initialize $s_0:=0$. In Algorithm~\ref{AlgsIP}, select the trial step-size at iteration $k$ as $\bar\lambda_k:=\theta_k\in(0,1)$.
Compute $\ell_k$, set the accepted step-size $
\lambda_k:=\beta^{\ell_k}\bar\lambda_k\in(0,1)$, and update $p_{k+1}:=\exp_{p_k}(\lambda_k\,v_k)$.
Update the counter and the next trial step-size by
\begin{equation*}
 \bar\lambda_{k+1}:=\theta_{k+1},   \quad \text{where}\quad  \theta_{k+1}:=
\begin{cases}
\max\{\underline\lambda,\ \beta^{s_{k+1}}\theta_0\}, & \text{if } \ell_k>0,\\[0.3em]
\theta_0, & \text{if } \ell_k=0,
\end{cases}
\qquad \text{and} \quad s_{k+1}:=s_k+\ell_k.
\end{equation*}
Then $\bar\lambda_k=\theta_k\in[\underline\lambda,\theta_0]\subset(0,1)$ for all $k$. Consequently, Assumption~{\bf (A2)} is satisfied. Furthermore, since $\bar\lambda_k\in(0,1]$ and $\lambda_k=\beta^{\ell_k}\bar\lambda_k\in(0,1]$, every trial point
$\exp_{p_k}(\beta^\ell \bar\lambda_k\,v_k)$ belongs to the geodesic segment connecting $p_k$ and $z_k$. Therefore, by the convexity of $C$, all function evaluations involved in~\eqref{eq:jkIP}, as well as the next iterate $p_{k+1}$, remain feasible.
\end{enumerate}

The trial-step-size rules above are all instances of the same backtracking template~\eqref{eq:jkIP}. They differ only in how the trial step-size $\bar\lambda_k\in(0,1]$ is selected before testing the geometric candidates $\beta^\ell\bar\lambda_k$. The trial step-size selection in \ref{it:ArmijoIP} corresponds to the fixed choice $\bar\lambda_k\equiv 1$, so the line search starts from the full feasible step along the geodesic segment $s\mapsto \exp_{p_k}(s\,v_k)$. The trial step-size selection in \ref{it:AdaptLIP} chooses $\bar\lambda_k$ from an estimate of a Lipschitz/curvature constant (e.g., $\bar\lambda_k=1/L_{k-1}\in(0,1]$) and updates this estimate after the backtracking test. Under the Lipschitz model, this rule yields $\inf_k\bar\lambda_k>0$ and thus also a uniform lower bound on the accepted step sizes. The trial step size selection in \ref{it:AdaptIP} is obtained by reusing the previously accepted step size and attempting one backtracking level larger; that is,
\[
\bar\lambda_k=\max\Bigl\{\underline\lambda,\ \min\{1,\beta^{-1}\lambda_{k-1}\}\Bigr\},
\qquad\text{with}\qquad
\lambda_{-1}:=\bar\lambda_0\in(0,1].
\]
In this way, we ensure that $\bar\lambda_k \in [\underline\lambda,1]$. Consequently, under rule (St3), we also have $\inf_k \bar\lambda_k > 0$. The trial step-size selection in \ref{it:St3IP} chooses $\bar\lambda_k=\theta_k$ and updates $\theta_{k+1}$ through a counter, with truncation at a prescribed $\underline\lambda>0$. This yields $\bar\lambda_k\in[\underline\lambda,1]$ for all $k$ and hence $\inf_k\bar\lambda_k>0$ by construction. Therefore, in all cases we have
$ \inf_k \bar\lambda_k > 0$ and $\sup_k \bar\lambda_k \leq 1$.
This guarantees that every trial point
$
\exp_{p_k}(\beta^\ell \bar{\lambda}_k v_k)
$
belongs to the geodesic segment joining \(p_k\) and \(z_k\). Since \(C\) is  convex and \(p_k,z_k\in C\), all trial points generated during the line-search procedure remain feasible.

Finally, Algorithm~\ref{AlgsIP} allows any choice of trial step-sizes satisfying $\bar\lambda_k\in(0,1]$, thereby providing greater flexibility than the standard restart strategy $\bar\lambda_k\equiv 1$. In favorable situations, this additional flexibility may lead to larger accepted step sizes and, consequently, to a reduction in the number of function evaluations required by the line-search procedure.

\subsubsection{Choice of step-size}\label{step-size}

We consider two rules for selecting the parameter \(t_k\). The first one is a
constant rule, in which
\[
    t_k \equiv t_0,
    \qquad k\geq 0,
\]
for some fixed \(t_0\in[t_{\min},t_{\max}]\). 
The second choice is based on a Riemannian Barzilai--Borwein rule, i.e., $ t_k    =  \min\left\{
        t_{\max},
        \max\left\{
            t_{\min},
            t_k^{\mathrm{BB}}
        \right\}
    \right\}
    $ where 

\begin{equation*}
t^{\mathrm{BB}}_{k} :=
\begin{cases}
\displaystyle
1,
&  \text{if }  k=0,\\[2ex]
\frac{
\langle s_{k-1},s_{k-1}\rangle
}{
\langle s_{k-1},y_{k-1}\rangle
},
& \text{if } \langle s_{k-1},y_{k-1}\rangle>0 \text{ and } k>0, \\[2ex]
t_{\max},
& \text{otherwise},
\end{cases}
\end{equation*}
with
$
    s_{k-1}:=-\log_{p_k}p_{k-1}$ and 
    $y_{k-1}
    :=
    \operatorname{grad} f(p_k)
    -
    P_{p_{k-1}p_k}
    \operatorname{grad} f(p_{k-1}),
$

Although this rule requires only first-order information, it implicitly incorporates second-order information through a scalar approximation of the Riemannian Hessian. This Riemannian Barzilai--Borwein strategy was introduced and studied for unconstrained optimization problems on Riemannian manifolds; see \cite{IannazzoPorcelli2018}.

%%%%%%%%%%%%%%%%%%%%%%%%%%%%%%%%%%
\subsubsection{Well-Definedness}\label{sec:wdef}  

In this section, we establish that Algorithm~\ref{AlgsIP} is well defined. In particular, whenever \(p_k\) is not a stationary point, the set defined in~\eqref{eq:jkIP} is nonempty, and hence \(\ell_k\), \(\lambda_k\), and \(p_{k+1}\) are well defined. Furthermore, since \(\bar{\lambda}_k \in (0,1]\), all trial points lie on the geodesic segment joining \(p_k\) and \(z_k \in C\). As \(C\) is convex, then \(p_{k+1}\in C\).

\begin{proposition}\label{pr:wdAstsIP}
Let $k\in\mathbb N$. Assume that $p_k\in C$ is not a stationary point of Problem~\eqref{eq:OptP}.
Then the iterate $p_{k+1}$ generated by Algorithm~\ref{AlgsIP} is well defined and belongs to $C$.
\end{proposition}

\begin{proof}
Recall that $\rho,\beta\in(0,1)$ and that the Algorithm~\ref{AlgsIP} chooses a trial step-size ${\bar\lambda}_k\in(0,1]$. Because $C$ is closed and  convex, the metric projection
$z_k={\cal P}_C(y_k)$ is well defined and belongs to $C$, see Lemma~\ref{Lem:nonexp.int.neg}$(i)$.
Since $\mathbb M$ is a  Hadamard manifold, the logarithm map is globally defined, hence
\(
v_k=\log_{p_k}z_k\in T_{p_k}\mathbb M
\)
is well defined. Since $p_k$ is not stationary, Proposition~\ref{pr:ProjProperty}$(ii)$ yields
\begin{equation}\label{eq:descent-dir-wdAstsIP}
\langle \grad f(p_k),v_k\rangle <0.
\end{equation}
Define $\phi:[0,1]\to\mathbb R$ by
$
\phi(\alpha):= f \left(\exp_{p_k}(\alpha\,v_k)\right).
$
Since $f$ is $C^1$ and $\exp$ is smooth, $\phi$ is differentiable and
\(
\phi'(0)=\langle \grad f(p_k),v_k\rangle.
\)
Hence, we obtain
\[
\lim_{\alpha\downarrow 0}\left( \frac{\phi(\alpha)-\phi(0)}{\alpha}-\rho\,\langle \grad f(p_k),v_k\rangle\right)
=(1-\rho)\langle \grad f(p_k),v_k\rangle<0,
\]
where the strict negativity follows from~\eqref{eq:descent-dir-wdAstsIP}.
Then, there exists $\delta>0$ such that for all $\alpha\in(0,\delta]$,
\[
\frac{f(\exp_{p_k}(\alpha v_k))-f(p_k)}{\alpha}
\le \rho\,\langle \grad f(p_k),v_k\rangle.
\]
Now choose $\hat\ell\in\mathbb N$ such that $\beta^{\hat\ell}{\bar\lambda}_k\le \delta$.
Then for every $\ell\ge \hat\ell$ we have $\alpha=\beta^\ell{\bar\lambda}_k\in(0,\delta]$, and therefore
\[
f\big(\exp_{p_k}(\beta^\ell{\bar\lambda}_k\, v_k)\big)
\le f(p_k)+\rho\,\beta^\ell{\bar\lambda}_k\,\langle \grad f(p_k),v_k\rangle.
\]
This shows that the set in~\eqref{eq:jkIP} is nonempty, and hence $\ell_k$ is well defined. 
Since $v_k=\log_{p_k}z_k$, the curve $\gamma(t):=\exp_{p_k}(t v_k)$, $t\in[0,1]$, is the unique minimizing
geodesic from $p_k$ to $z_k$, and $\gamma(\lambda_k)=p_{k+1}$.
Because $p_k\in C$, $z_k\in C$, and $C$ is convex, we have $\gamma([0,1])\subset C$, hence
$p_{k+1}\in C$. Therefore, $p_{k+1}$ is well defined and belongs to $C$.
\end{proof}

From this point on, we assume that the algorithm does not stop, that is, \(z_k\neq p_k\) for all \(k\in\mathbb N\). By Proposition~\ref{pr:ProjProperty}\textup{$(i)$}, this implies that \(p_k\) is not a stationary point of~\eqref{eq:OptP} for any \(k\). Consequently, using~\eqref{eq:cpIP} and Proposition~\ref{pr:ProjProperty}\textup{$(ii)$}, we obtain
\begin{equation}\label{eq:neg-dire}
\langle \grad f(p_k),v_k\rangle<0,
\qquad \forall\, k\in\mathbb N.
\end{equation}
Moreover, since $p_0\in C$, Proposition~\ref{pr:wdAstsIP} ensures that $\ell_k$ is well defined at every iteration and that the
sequence $\{p_k\} \subset C$ generated by Algorithm~\ref{AlgsIP} is infinite.

%%%%%%%%%%%%%%%%%%%%%%%%%%%%%%%%%%%%%%%%%%%%%%%%%%%%%%%%%%%%%%%%%%%%%%%%
\subsubsection{Asymptotic convergence analysis}  \label{sec:AsyConAnal}

In this section, we study the asymptotic behavior of the sequence
\(\{p_k\}\) generated by Algorithm~\ref{AlgsIP}, assuming that the stopping criterion is never satisfied. Our goal is to prove that every accumulation point of \(\{p_k\}\) is a stationary point of problem~\eqref{eq:OptP}. The next lemma establishes the basic asymptotic properties required for the convergence analysis. 

\begin{lemma}\label{lem:desc-mono}
Let $\{p_k\}$ be the sequence generated by Algorithm~\ref{AlgsIP}. Then the sequence of function values $\{f(p_k)\}$ is decreasing and convergent. Moreover,
\begin{equation}\label{eq:fcddIP}
\lim_{k\to+\infty}
\lambda_k\,\langle \grad f(p_k), v_k\rangle
=
0.
\end{equation}
\end{lemma}
\begin{proof}
By \eqref{eq:pk+1thet} and \eqref{eq:jkIP}, the Armijo condition is satisfied at the accepted step size \(\lambda_k\). Hence,

\begin{equation}\label{eq:arm-acept-dsc}
f(p_{k+1})
=
f\!\left(\exp_{p_k}(\lambda_k v_k)\right)
\le
f(p_k)
+
\rho\,\lambda_k\,\langle \grad f(p_k),v_k\rangle, \qquad \forall k\in\mathbb{N}.
\end{equation}
Since \(p_k\) is assumed to be nonstationary for every \(k\), it follows from \eqref{eq:neg-dire} that
$
\langle \grad f(p_k),v_k\rangle < 0.
$
Therefore, \eqref{eq:arm-acept-dsc} implies that
$
f(p_{k+1}) < f(p_k)$ for all
$k\in\mathbb{N}$,
and hence the sequence \(\{f(p_k)\}\) is monotonically decreasing. Since \(\Omega^\ast \neq \varnothing\), the objective function is bounded below on $C$. Consequently, \(\{f(p_k)\}\) converges. On the other hand, combining \eqref{eq:neg-dire} with \eqref{eq:arm-acept-dsc} yields
\[
0
<
-\rho\,\lambda_k
\langle \grad f(p_k),v_k\rangle
\le
f(p_k)-f(p_{k+1}), \qquad \forall k\in\mathbb{N}.
\]
Taking limits as
\(k\to+\infty\) and using the convergence of \(\{f(p_k)\}\), we obtain
\[
\lim_{k\to+\infty}
\Bigl(
-\rho\,\lambda_k
\langle \grad f(p_k),v_k\rangle
\Bigr)
=0.
\]
Since \(\rho>0\), we conclude that
$
\lim_{k\to+\infty}
\lambda_k
\bigl\langle \grad f(p_k),v_k\bigr\rangle
=0.
$
This completes the proof.
\end{proof}

We are now in a position to establish the main convergence result of this section. The following theorem shows that every accumulation point of the sequence generated by Algorithm~\ref{AlgsIP} satisfies the first-order stationarity condition for problem~\eqref{eq:OptP}.

\begin{theorem}\label{teo.MainIP} 
Assume that \textbf{(A2)} holds. Let $\bar p \in C$ be an accumulation point of the sequence 
$\{p_k\}$ generated by Algorithm~\ref{AlgsIP}. 
Then $\bar p$ is a stationary point of Problem~\eqref{eq:OptP}.
\end{theorem}

\begin{proof}
Let \(\bar p\) be an accumulation point of the sequence \(\{p_k\}\). If
\(\grad f(\bar p)=0\), then \(\bar p\) is stationary by Definition~\ref{def:sta.point}. Then,  consider the case \(\grad f(\bar p)\neq 0\). Choose a subsequence  $\{p_{k_j}\}$  for $\{p_k\}$ such that $\lim_{j\to+\infty}p_{k_j}=\bar p$.
Since $\{t_k\}\subset[t_{\min},t_{\max}]$ by \eqref{eq:coaIP}, 
 $\{ \bar\lambda_k \} \subset [\lambda^{\star}, 1]$  
and $\{\lambda_k\} \subset(0,1]$, it follows that, up to a subsequence (not relabeled), we may assume
\begin{equation*}\label{eq:tmbsub}
\lim_{j\to+\infty} t_{k_j} = \tilde t \in [t_{\min},t_{\max}],
\qquad
\lim_{j\to+\infty} \bar{\lambda}_{k_j} = \hat\lambda \in [\lambda^{\star}, 1],
\qquad
\lim_{j\to+\infty} \lambda_{k_j} = \tilde\lambda \in [0,1].
\end{equation*}
By the continuity of the gradient vector field and the exponential map, we have
$$
\lim_{j\to+\infty}y_{k_j}=\lim_{j\to+\infty}\exp_{p_{k_j}}(-t_{k_j}\grad f(p_{k_j}))=\exp_{\bar p}(-\tilde t\,\grad f(\bar p)).
$$
Furthermore, by the continuity of the logarithm map and the projection map, it follows that
\begin{equation*}\label{eq:zkj-limit-mainIP}
\lim_{j\to+\infty}z_{k_j}=\lim_{j\to+\infty}{\cal P}_C(y_{k_j})= \bar z, \quad \qquad \lim_{j\to+\infty}v_{k_j}=\lim_{j\to+\infty}\log_{p_{k_j}}z_{k_j}= \log_{\bar p}\bar z,
\end{equation*} 
where $\bar z:={\cal P}_C\!\left(\exp_{\bar p}\!\bigl(-\tilde t\,\grad f(\bar p)\bigr)\right)$. We analyse two cases:

\noindent\emph{Case 1: $\tilde\lambda>0$.} From \eqref{eq:fcddIP} and $\lim_{j\to+\infty}\lambda_{k_j}=\tilde\lambda>0$, we  obtain that 
$ \lim_{j\to+\infty}\langle \grad f(p_{k_j}),v_{k_j}\rangle=0$.
Then, Lemma~\ref{le:coslawap} applies with $p=p_{k_j}$, $\alpha=t_{k_j}$ and $v=\grad f(p_{k_j})$ yields 
$$
0 \leq d^2(p_{k_j},z_{k_j})
\le -\,t_{k_j}\big\langle \grad f(p_{k_j}),v_{k_j}\big\rangle,
$$
which implies that $\lim_{j\to+\infty}d(p_{k_j},z_{k_j})= 0$.
Hence, $  \lim_{j\to+\infty}z_{k_j}=\bar p$, and therefore 
\[
\bar p = \bar z
={\cal P}_C\!\left(\exp_{\bar p}\!\bigl(-\tilde t\,\grad f(\bar p)\bigr)\right).
\]
Hence, it follows from Proposition~\ref{pr:ProjProperty}$(i)$ that \( \bar p \) is a stationary point of problem~\eqref{eq:OptP}.

\smallskip
\noindent\emph{Case 2: $\tilde\lambda=0$.}

Since \(\lim_{j\to+\infty}\beta^{\ell_{k_j}}\bar\lambda_{k_j}=\lim_{j\to+\infty}\lambda_{k_j}=0\) and \(\lim_{j\to+\infty}\bar\lambda_{k_j}=\hat\lambda\in(0,1]\), it follows that \(\lim_{j\to+\infty}\ell_{k_j}=+\infty\). Fix an arbitrary \(m\in\mathbb N\). Since \(\lim_{j\to+\infty}\ell_{k_j}=+\infty\), for all sufficiently large \(j\) we have \(m<\ell_{k_j}\). Hence, by the minimality of \(\ell_{k_j}\) in~\eqref{eq:jkIP}, the Armijo condition fails at \(\ell=m\), and therefore
\[
\frac{f\big(\exp_{p_{k_j}}\big(\beta^m\bar\lambda_{k_j}v_{k_j}\big)\big)-f(p_{k_j})}{\beta^m\bar\lambda_{k_j}}>\rho\,\big\langle\grad f(p_{k_j}),v_{k_j}\big\rangle.
\]
For this fixed \(m\), taking the limit as \(j\to+\infty\), and using \(\lim_{j\to+\infty}p_{k_j}=\bar p\), \(\lim_{j\to+\infty}\bar\lambda_{k_j}=\hat\lambda\), and \(\lim_{j\to+\infty}z_{k_j}=\bar z\), together with the continuity of the logarithm map as a map into the tangent bundle, we obtain
\[
\frac{f\big(\exp_{\bar p}\big(\beta^m\hat\lambda\log_{\bar p}\bar z\big)\big)-f(\bar p)}{\beta^m\hat\lambda}\ge\rho\,\big\langle\grad f(\bar p),\log_{\bar p}\bar z\big\rangle.
\]
Now, taking the limit as \(m\to+\infty\), and noting that \(\lim_{m\to+\infty}\beta^m\hat\lambda=0\), differentiability of \(f\) along the geodesic \(s\mapsto\exp_{\bar p}(s\log_{\bar p}\bar z)\) yields \(\big\langle\grad f(\bar p),\log_{\bar p}\bar z\big\rangle\ge\rho\,\big\langle\grad f(\bar p),\log_{\bar p}\bar z\big\rangle\). Since \(0<\rho<1\), it follows that
\begin{equation}\label{eq:nonneg-inner-mainIP}
\big\langle\grad f(\bar p),\log_{\bar p}\bar z\big\rangle\ge0.
\end{equation}
On the other hand, applying Lemma~\ref{le:coslawap} with \(p=\bar p\), \(\alpha=\tilde t\), and \(v=\grad f(\bar p)\), we obtain
\[
\big\langle\grad f(\bar p),\log_{\bar p}\bar z\big\rangle\le-\frac{1}{\tilde t}d^2(\bar p,\bar z)\le0.
\]
Combining this inequality with \eqref{eq:nonneg-inner-mainIP} yields \(d(\bar p,\bar z)=0\), and hence \(\bar z=\bar p\). Therefore, \(\bar p={\cal P}_C\!\left(\exp_{\bar p}\!\bigl(-\tilde t\,\grad f(\bar p)\bigr)\right)\). It then follows from Proposition~\ref{pr:ProjProperty}\textup{$(i)$} that \(\bar p\) is a stationary point of problem~\eqref{eq:OptP}. Since the same conclusion holds in both cases, we conclude that \(\bar p\) is stationary, which completes the proof.

\end{proof}

We next show that, under a compactness assumption, the sequence generated by Algorithm~\ref{AlgsIP} admits accumulation points and that every such accumulation point is stationary.

\begin{corollary}\label{pr:cluster-points-existence}
Suppose that either \(C\) is compact or the lower level set
$$
L_0:=\{p\in C:\ f(p)\le f(p_0)\}
$$
is compact. Then the sequence \(\{p_k\}\) generated by Algorithm~\ref{AlgsIP} admits at least one accumulation point in \(C\). Moreover, if
\textbf{(A2)} holds
then every accumulation point of \(\{p_k\}\) is stationary for problem~\eqref{eq:OptP}. In particular, the sequence admits at least one stationary accumulation point.
\end{corollary}
\begin{proof}
By Lemma~\ref{lem:desc-mono}, the sequence $\{f(p_k)\}$ is decreasing, hence
\begin{equation}\label{eq:levelset-inclusion-cluster}
f(p_k)\leq f(p_0)\qquad \forall k\in\mathbb N,
\end{equation}
which implies $p_k\in L_0$ for all $k\in\mathbb{N}$. If $C$ is compact, then $\{p_k\} \subset C$ admits at least one accumulation point in \(C\).
On the other hand, if \(L_0\) is compact, then \eqref{eq:levelset-inclusion-cluster} yields \(\{p_k\}\subset L_0\), and consequently \(\{p_k\}\) admits at least one accumulation point in \(L_0\subseteq C\). This proves the existence of accumulation points. Now assume that \textbf{(A2)} holds,  and let \(\bar p\) be an arbitrary accumulation point of \(\{p_k\}\). Then, Theorem~\ref{teo.MainIP} implies that \(\bar p\) is a stationary point of problem~\eqref{eq:OptP}. Since at least one accumulation point exists, it follows that the sequence \(\{p_k\}\) possesses at least one stationary accumulation point.
\end{proof}

\begin{remark}\label{rk:convex-case-addon}
Assume, in addition, that \(f\) is geodesically convex on \(C\). Then every stationary point of problem~\eqref{eq:OptP} is a global minimizer of \(f\) over \(C\). Consequently, under the assumptions of Theorem~\ref{teo.MainIP}, every accumulation point of \(\{p_k\}\) is a global minimizer. Suppose now that \(f\) is strictly geodesically convex on \(C\). By the standing assumption \(\Omega^*\neq\varnothing\), problem~\eqref{eq:OptP} has a unique minimizer, say \(p^*\). We claim that the initial lower level set \(L_0:=\{p\in C:f(p)\le f(p_0)\}\) is compact. If \(f(p_0)=f(p^*)\), then \(L_0=\{p^*\}\). Otherwise, suppose by contradiction that \(L_0\) is unbounded. Then there exists \(q_j\in L_0\) such that \(r_j:=d(p^*,q_j)\to\infty\). For all sufficiently large \(j\), define \(u_j:=\exp_{p^*}(r_j^{-1}\log_{p^*}q_j)\). Since \(C\) is convex, \(u_j\in C\), and \(d(p^*,u_j)=1\). The unit sphere centered at \(p^*\) is compact in a finite-dimensional Hadamard manifold, so, after passing to a subsequence, \(u_j\to u\in C\) with \(d(p^*,u)=1\). Geodesic convexity gives \(f(u_j)\le(1-r_j^{-1})f(p^*)+r_j^{-1}f(q_j)\le f(p^*)+(f(p_0)-f(p^*))/r_j\). Passing to the limit yields \(f(u)=f(p^*)\), contradicting uniqueness because \(u\neq p^*\). Thus \(L_0\) is bounded. It is also closed, and hence compact by Hopf--Rinow. Since Lemma~\ref{lem:desc-mono} gives \(p_k\in L_0\) for every \(k\), the sequence \(\{p_k\}\) is precompact. Every accumulation point is a global minimizer and therefore equals \(p^*\). Hence the whole sequence converges to \(p^*\). The same conclusion holds, in particular, when \(f\) is strongly geodesically convex on \(C\).
\end{remark}

%%%%%%%%%%%%%%%%%%%%%%%%%%%%%%%%%%%%%%%%%%%%%%%%%%%%%%%%%%%%%%%%%%%%%%%%
\subsubsection{Complexity analysis}   \label{sec:ComplAnal}
In this section, we establish an iteration-complexity bound for the stationarity measure \(d(p_k,z_k)\), where \(z_k\) denotes the projected point computed by Algorithm~\ref{AlgsIP}. This measure is natural in the sense that \(d(p_k,z_k)=0\) if and only if \(p_k\) satisfies the first-order stationarity condition for problem~\eqref{eq:OptP}. Our goal is to derive an upper bound on \[ \min_{0\le k\le N} d(p_k,z_k) \] as a function of the iteration counter \(N\). The analysis relies on two key ingredients. First, we assume that the trial step-sizes are uniformly bounded away from zero, a condition that can be readily enforced in practice. Second, under Assumptions~\textbf{(A1)} and \textbf{(A2)}, and the bounds established in~\eqref{eq:coaIP}, we show that the accepted step-sizes also admit a uniform positive lower bound. Combining this property with the Armijo decrease condition and the projection inequality yields a telescoping argument, from which the standard complexity estimate of order \(O(N^{-1/2})\) follows.

\begin{lemma}\label{Le:MainConvPAstsIP}
Let $\{p_k\}$ be the sequence generated by Algorithm~\ref{AlgsIP}. Suppose that \textbf{(A1)} and~\textbf{(A2)} hold. Then, 
\begin{equation}\label{eq:lb-lambda}
\lambda_k
\ge
\bar{\lambda}
:=
\min\!\left\{
\lambda^{\star},
\frac{2\beta(1-\rho)}{t_{\max}\,L}
\right\},
\qquad \forall\, k\in\mathbb N.
\end{equation}
\end{lemma}
\begin{proof}
Fix \(k \in \mathbb{N}\). If \(\ell_k=0\), then, by the definition of $\lambda_k$ and  \textbf{(A2)}, we have
$\lambda_k=\bar{\lambda}_k\ge \lambda^{\star}$.
Therefore, $\lambda_k\ge \lambda^{\star}$.  Assume now that $\ell_k\ge 1$. In this case, 
$
\beta^{\ell_k-1}\bar\lambda_k=\lambda_k/\beta\in(0,1],
$
Since $\beta^{\ell_k-1}\bar\lambda_k\le 1$ and $z_k\in C$, the  convexity of $C$ implies $\exp_{p_k}(\beta^{\ell_k-1}\bar\lambda_k v_k) \in C$.
By the minimality of $\ell_k$ in \eqref{eq:jkIP}, the Armijo test fails at $\ell=\ell_k-1$, i.e.,
\begin{equation}\label{eq:armijo-fails-ak-lemma}
f(\exp_{p_k}(\beta^{\ell_k-1}\bar\lambda_k v_k)) > f(p_k)+\rho (\beta^{\ell_k-1}\bar\lambda_k)\langle \grad f(p_k),v_k\rangle.
\end{equation}
On the other hand,  \textbf{(A1)}, together with Lemma~\ref{le:lc} applied to \(q=\exp_{p_k}(\beta^{\ell_k-1}\bar{\lambda}_k v_k)\) and \(p=p_k\), yields
\begin{equation}\label{eq:lc-ak-lemma}
f(\exp_{p_k}(\beta^{\ell_k-1}\bar\lambda_k v_k)) \leq f(p_k)+ \beta^{\ell_k-1}\bar\lambda_k\langle \grad f(p_k),v_k\rangle+\frac{L}{2}(\beta^{\ell_k-1}\bar\lambda_k)^2\,\|v_k\|^2.
\end{equation}
Taking \(p=p_k\), \(\alpha=t_k\), and \(v=\grad f(p_k)\) in~\eqref{eq;snlpf}, we obtain
\begin{equation}\label{eq:proj-vk-bound-lemma}
\|v_k\|^2=d^2(p_k,z_k)\ \le\ -\,t_k\,\langle \grad f(p_k),v_k\rangle .
\end{equation}
By using \eqref{eq:lc-ak-lemma} with \eqref{eq:proj-vk-bound-lemma} one has
\begin{equation}\label{eq:lc-ak-lemma2}
f(\exp_{p_k}(\beta^{\ell_k-1}\bar\lambda_k v_k))  \leq f(p_k)+\beta^{\ell_k-1}\bar\lambda_k\left(1-\frac{L}{2}(\beta^{\ell_k-1}\bar\lambda_k) t_k\right)\langle \grad f(p_k),v_k\rangle.
\end{equation}
Combining \eqref{eq:armijo-fails-ak-lemma} with \eqref{eq:lc-ak-lemma2} and \eqref{eq:neg-dire}, we obtain  $1- L/2(\beta^{\ell_k-1}\bar\lambda_k) t_k < \rho$, which implies that 
\[
\beta^{\ell_k-1}\bar\lambda_k > \frac{2(1-\rho)}{L\,t_k}\ \ge\ \frac{2(1-\rho)}{L\,t_{\max}},
\]
where the last inequality follows from \eqref{eq:coaIP}.
Multiplying by $\beta$ gives
$
\lambda_k= \beta^{\ell_k}\bar\lambda_k  > 2\beta(1-\rho)/(L\,t_{\max}).
$
Therefore, in the case \(\ell_k \ge 1\), we have
$
\lambda_k \ge 2\beta(1-\rho)/(Lt_{\max}).
$
Hence, inequality~\eqref{eq:lb-lambda} holds in both cases, which completes the proof.
\end{proof}
Next, we establish a uniform bound for the partial sums \(\sum_{k=0}^{N} d^2(p_k,z_k)\).
\begin{lemma}\label{lem:telescoping-dist2}
Suppose that \textbf{(A1)} and~\textbf{(A2)} hold.
Let $\{p_k\}$ and $\{z_k\}$ be generated by Algorithm~\ref{AlgsIP}.
Then, for every $N\in\mathbb N$, there holds 
\begin{equation}\label{eq:sum-dist2-bound}
\sum_{k=0}^{N} d^2(p_k,z_k)
\ \le\
\frac{t_{\max} \bigl(f(p_0)-f(p_{N+1})\bigr)}{\rho\,\bar\lambda}
\ \le\
\frac{t_{\max}\bigl(f(p_0)-f^*\bigr)}{\rho\,\bar\lambda}\,,
\end{equation}
where $\bar{\lambda}$ is defined in \eqref{eq:lb-lambda}.
\end{lemma}

\begin{proof}
Fix $k\in\mathbb N$. By \eqref{eq:jkIP} and \eqref{eq:pk+1thet}, we have
$$
f(p_{k+1}) - f(p_k) \le \rho\lambda_k \langle \grad f(p_k),v_k\rangle.
$$
On the other hand, by \eqref{eq;snlpf} with $\alpha =t_k$, $p=p_k$ and $v= \grad f(p_k)$, we have
\[
\langle \grad f(p_k),v_k\rangle\ \le\ -\frac{1}{t_k}\,d^2(p_k,z_k).
\]
By combining the previous two inequalities and applying Lemma~\ref{Le:MainConvPAstsIP} and \eqref{eq:coaIP}, we obtain
\[
f(p_k)-f(p_{k+1})
\ge
\rho\,\frac{\lambda_k}{t_k}\,d^2(p_k,z_k) \geq \rho\,\frac{ \bar{\lambda} }{t_{\max}}\,d^2(p_k,z_k) .
\]
Summing  from $k=0$ to $k=N$ gives
$$
\frac{\rho\,\bar\lambda}{t_{\max}}\sum_{k=0}^{N} d^2(p_k,z_k)
 \le
\sum_{k=0}^{N}\bigl(f(p_k)-f(p_{k+1})\bigr) 
 =
f(p_0)-f(p_{N+1}) \leq f(p_0)- f^*,
$$
which proves the first inequality in \eqref{eq:sum-dist2-bound}. 
\end{proof}

Since the quantity $d(p_k,z_k)$ can be interpreted as a measure of stationarity of $p_k$, the next result provides an iteration-complexity bound for this measure.

\begin{theorem}\label{Teo:FCompP2}
Suppose that \textbf{(A1)} and~\textbf{(A2)} hold. Let $\{p_k\}$ and  $\{z_k\}$  be generated by Algorithm~\ref{AlgsIP}.  Then, for every $N\in\mathbb N$,
\begin{equation}\label{eq:s-dt2-bd-prf}
\min\bigl\{\,d(p_k,z_k): k=0,1,\ldots,N\,\bigr\}
 \le
\sqrt{\frac{t_{\max}\bigl(f(p_0)-f^*\bigr)}{\rho\,\bar \lambda}}\frac{1}{\sqrt{N+1}}.
\end{equation}
Consequently, for any \(\varepsilon>0\), if
\begin{equation}\label{eq:s-dt22-bd-prf}
N  \ge \left\lceil \frac{t_{\max}\bigl(f(p_0)-f^*\bigr) }{\rho\,\bar\lambda\,\varepsilon^2} \right\rceil-1=:N_{\varepsilon},
\end{equation}
then there exists \(k\in\{0,\ldots,N\}\) such that
\(
d(p_k,z_k)\le \varepsilon.
\)
Here, \(\lceil\cdot\rceil\) denotes the ceiling function.
\end{theorem}

\begin{proof}
By Lemma~\ref{lem:telescoping-dist2}, for every \(N\in\mathbb{N}\), it holds that
\begin{equation}\label{eq:s-p12-bd-prf}
(N+1)\min_{0\le k\le N} d^2(p_k,z_k)
\le
\sum_{k=0}^{N} d^2(p_k,z_k)
\le
\frac{t_{\max}}{\rho\,\bar\lambda}\bigl(f(p_0)-f^*\bigr).
\end{equation}
Taking square roots immediately yields
\eqref{eq:s-dt2-bd-prf}.

Let $\varepsilon>0$ be arbitrary and suppose, by contradiction, that $d(p_k,z_k)>\varepsilon$ for every $k\in\{0,1,\ldots,N_\varepsilon\}$. Then, by the definition of \(N_\varepsilon\) in~\eqref{eq:s-dt22-bd-prf}, we obtain
\[
\sum_{k=0}^{N_\varepsilon} d^2(p_k,z_k)>(N_\varepsilon+1)\varepsilon^2=\left\lceil\frac{t_{\max}\bigl(f(p_0)-f^*\bigr)}{\rho\,\bar\lambda\,\varepsilon^2}\right\rceil\varepsilon^2\ge\frac{t_{\max}\bigl(f(p_0)-f^*\bigr)}{\rho\,\bar\lambda}.
\] 
This contradicts the second inequality in~\eqref{eq:s-p12-bd-prf} with \(N=N_\varepsilon\). Consequently, there exists \(k\in\{0,1,\ldots,N_\varepsilon\}\) such that $d(p_k,z_k)\le \varepsilon.$ This completes the proof.
\end{proof}

\begin{remark}\label{rk:lambda-lb-implement}
Under Assumptions~\textbf{(A1)}--\textbf{(A2)} and the bounds in~\eqref{eq:coaIP}, Lemma~\ref{Le:MainConvPAstsIP} ensures that
\[
\lambda_k=\beta^{\ell_k}\bar\lambda_k
\geq
\bar{\lambda}
:=
\min\!\left\{
\lambda^{\star},
\frac{2\beta(1-\rho)}{t_{\max}\,L}
\right\},
\qquad
\forall k\in\mathbb{N}.
\]
Consequently, the step sizes are uniformly bounded away from zero. As a result, the summability of the squared step lengths follows, and the \(O(1/\sqrt{N})\) stationarity-complexity bound established in Theorem~\ref{Teo:FCompP2} holds with explicit constants depending on \(\rho\), \(\beta\), \(t_{\max}\), \(L\), \(\lambda^\star\), and the initial optimality gap \(f(p_0)-f^*\).

\end{remark}

Next, we establish a square-summability property for the successive displacements
\(d(p_{k+1},p_k)\). In particular, this result yields a complexity estimate on the number of iterations for which the displacement exceeds a prescribed threshold.

\begin{corollary}\label{cor:asymptotic-regularity}
Under the assumptions of Theorem~\ref{Teo:FCompP2}, the sequence
\(\{d^2(p_{k+1},p_k)\}\) is summable. Consequently,
$\lim_{k\to+\infty} d(p_{k+1},p_k)=0$.
Moreover, for every \(\varepsilon>0\),
\[
\#\Bigl\{\,k\in\mathbb N:\ d(p_{k+1},p_k)\ge \varepsilon\,\Bigr\}
\le
\frac{t_{\max}}{\rho\,\bar\lambda}\,
\frac{f(p_0)-f^*}{\varepsilon^2}.
\]
Here, \(\#(\cdot)\) denotes the cardinality of the corresponding set.
\end{corollary}

\begin{proof}
By using \eqref{eq:pk+1thet}, we have $p_{k+1}=\exp_{p_k}(\lambda_k v_k)$ with $\lambda_k\in(0,1]$. Hence, we  conclude that  
\[
d(p_{k+1},p_k)=\|\lambda_k v_k\|=\lambda_k\|v_k\|
=\lambda_k\,d(p_k,z_k)\ \le\ d(p_k,z_k), \qquad \forall k\in \mathbb{N}.
\]
Squaring both sides and summing over \(k=0,\ldots,N\), and then invoking Lemma~\ref{lem:telescoping-dist2}, we obtain
\begin{equation}\label{eq:slity_stps}
\sum_{k=0}^{N} d^2(p_{k+1},p_k)
\le
\sum_{k=0}^{N} d^2(p_k,z_k)
\le
\frac{t_{\max}}{\rho\,\bar\lambda}\bigl(f(p_0)-f^*\bigr),
\qquad \forall\, N\in\mathbb{N}.
\end{equation}
It follows from \eqref{eq:slity_stps} that 
\(\{d^2(p_{k+1},p_k)\}\) is summable. Consequently,   $\lim_{k\to+\infty} d(p_{k+1},p_k)=0$, which proves the first statement.  Now fix \(\varepsilon>0\) and define $ \mathcal K_\varepsilon := \bigl\{ k\in\mathbb N: d(p_{k+1},p_k)\ge\varepsilon \bigr\}$. Then,
\[
\varepsilon^2\,\#\mathcal K_\varepsilon
\ \le\
\sum_{k\in\mathcal K_\varepsilon} d^2(p_{k+1},p_k)
\ \le\
\sum_{k=0}^{+\infty} d^2(p_{k+1},p_k)
\ \le\
\sum_{k=0}^{+\infty} d^2(p_k,z_k)\leq 
\frac{t_{\max}}{\rho\,\bar\lambda}\bigl(f(p_0)-f^*\bigr),
\]
where the last inequality follows from
\eqref{eq:slity_stps} by letting \(N\to\infty\).
This completes the proof.
\end{proof}

%%%%%%%%%%%%%%%%%%%%%%%%%%%%%%%%%%%%%%%%%%%%%%%%%%%%%%%%%%%%%%%%%%%%%%%%%%%%%%%%%%%%%%%%%%%%%%%%%%%%%%%%%%%%%%%%%%%%%%%%%%%%%%%%%%%%%%%%%%%%%%%%%%%%%%%%%%%%%%%%%%%%%%%%%%%%%%%%%%%%%%%%%%%%%%%%%%%%%%%%%%%%%%%%%%%%%%%%%%%%%%%%%%%%%%%%%%%%%%%%%%%%%%%%%%%%%%%%%%%%%%%%%%%%%%%%%%%%%%%%%%%%%%%%%%%%%%%%%%%%%%%%%%%%%%%%%%%%%%%%%%%%

%%%%%%%%%%%%%%%%%%%%%%%%%%%%%%%%%%%%%%%%%%%%%%%%%%%%%%%%%%%%%%%%%%%%%%%%%%%%%%%%%%%%%%%%%%%%%%%%%%%%%%%%%%%%%%%%%%%%%%%%%%%%%%%%%%%%%%%%%%%%%%%%%%%%%%%%%%%%%%%%%%%%%%%%%%%%%%%%%%%%%%%%%%%%%%%%%%%%%%%%%%%%%%%%%%%%%%%%%%%%%%%%%%%%%%

\section{Numerical experiments}\label{sec:nume.exp}

In this section, we present numerical experiments to illustrate the practical behavior of the projected gradient method on Hadamard manifolds. We consider a constrained optimization problem on the SPD manifold $\mathbb{P}^n$. 
%The problem consists of computing the Karcher mean over a lower-bounded spectral set. 
These experiments are not intended to constitute an exhaustive computational study. Instead, they are designed to investigate the numerical behavior of the proposed method  and to examine the influence of the step-size strategies introduced in this work.

We consider two variants of Algorithm~\ref{AlgsIP}, namely ALG2-C and
ALG2-BB, which employ constant step sizes ($t_k=0.01$ for all $k \geq 0$), 
a value selected through preliminary tuning, and Barzilai--Borwein step sizes
(see Section~\ref{step-size}), respectively, in Step~2.
The algorithmic parameters are set to
\[
\rho=10^{-4},\qquad \beta=0.5,\qquad
t_{\min}=10^{-4},\qquad
t_{\max}=10^{2},
\]
and the maximum number of line-search iterations is fixed at \(250\).  {If the Armijo condition was not satisfied within 250 backtracking reductions, the corresponding run was declared unsuccessful.}

Combining each variant with the four step-size strategies St1--St4 yields
the eight implementations
\[
\text{ALG2-C-St1},\ldots,\text{ALG2-C-St4},
\qquad
\text{ALG2-BB-St1},\ldots,\text{ALG2-BB-St4},
\]
where St2 uses \(L_0=1\), St3 uses \(\bar\lambda_0=1\) and \(\underline\lambda=10^{-5}\), and St4 uses \(\theta_0=0.9\) and  \(\underline\lambda=10^{-5}\).

The algorithm was terminated when either \(d(z_k,p_k)\leq 10^{-6}\) or the
maximum iteration limit of 200 was reached. Since the computation of the
Riemannian exponential and logarithm maps in \eqref{eq:exp.log.matrix}, together
with the projections onto the feasible sets considered below, relies on spectral
decompositions, we use the cumulative number of spectral decompositions
performed up to iteration \(k\), denoted by \(\texttt{NumberDec}(k)\), as the
primary measure of computational cost.

All algorithms were implemented in Python~3.14 and executed on a laptop equipped with an AMD Ryzen~5~7520U processor and 8~GB of LPDDR5-5500 RAM, running 64-bit Windows~11 Home. Numerical linear algebra operations were performed using NumPy~2.4.3 and SciPy~1.17.1, and all computations were carried out in double-precision arithmetic.

Let \(A_1,\ldots,A_m\in\mathbb{P}^n\). For our test problem,  we consider  the constrained Karcher
mean problem
\[
\min_{X\in C} f(X)
:=
\frac{1}{2}\sum_{i=1}^m d^2(X,A_i),
\qquad
C:=\left\{X\in\mathbb{P}^n:X\succeq\frac{1}{2}I\right\},
\]
where \(d\) denotes the affine-invariant Riemannian distance defined in
\eqref{eq:spd-distance}. The Riemannian gradient of \(f\) is given by
\[
\operatorname{grad} f(X)
=
-\sum_{i=1}^m\log_X(A_i),
\qquad X\in\mathbb{P}^n.
\]
We consider the following ten problem-size configurations:
\[
\begin{aligned}
(n,m)\in\{&
(10,100),\,
(15,75),\,
(20,50),\,
(30,25),\,
(40,20),\\
&
(100,50),\,
(125,40),\,
(150,30),\,
(175,20),\,
(200,10)
\}.
\end{aligned}
\]
For each pair \((n,m)\), we generate ten feasible initial points of the form
$
    X_0^{(r)} = 0.5 r I$,  $r = 1,\ldots,10,
$
yielding a total of 100 test problem instances for each class of the well- and ill-conditioned test sets. By Proposition~\ref{prop:spd-projection}, the metric projection onto \(C\) is obtained by spectral clipping. Specifically, if \(Y=Q\operatorname{diag}(\lambda_1,\ldots,\lambda_n)Q^\top\in\mathbb P^n\), then
\[
{\cal P}_C(Y)=Q\operatorname{diag}\bigl(\max\{\lambda_1,0.5\},\ldots,\max\{\lambda_n,0.5\}\bigr)Q^\top.
\]
The data matrices \(A_1,\ldots,A_m\) were generated as follows. For each \(i=1,\ldots,m\), we first sample a random orthogonal matrix \(Q_i \in \mathbb{R}^{n \times n}\). Given a prescribed condition number \(\kappa_{\mathrm{data}}\), we set
$
    \lambda_{\min} = \kappa_{\mathrm{data}}^{-1/2},
$ and $
    \lambda_{\max} = \kappa_{\mathrm{data}}^{1/2}$,
so that the geometric mean of the extreme eigenvalues is equal to one.
The remaining \(n-2\) eigenvalues are sampled independently in logarithmic scale over the interval \([\lambda_{\min}, \lambda_{\max}]\). 
Each matrix is then constructed as
\[
    A_i
    =
    Q_i \operatorname{diag}\!\bigl(\lambda_1^{(i)},\ldots,\lambda_n^{(i)}\bigr) Q_i^\top,
    \qquad i=1,\ldots,m,
\]
with \(\lambda_1^{(i)} = \lambda_{\min}\) and \(\lambda_n^{(i)} = \lambda_{\max}\). Consequently, each \(A_i\) is symmetric positive definite with condition number \(\kappa_{\mathrm{data}}\).

We consider two groups of test instances. The first group is formed by well-conditioned data matrices, with $\kappa_{\mathrm{data}}=10^2$, whereas the second group is formed by ill-conditioned data matrices, with $\kappa_{\mathrm{data}}=10^5$. The solvers are compared by means of performance profiles in the sense of Dolan and Mor\'e~\cite{dolan2002benchmarking}.  In the ill-conditioned case, we reset $t_{\max}=10$ for the ALG2-BB variants.

The performance profiles, based on the $\texttt{NumberDec}(k)$ metric and CPU time, for the well-conditioned test set are presented in Figure~\ref{fig:karcher-well-performance-profiles}. According to the $\texttt{NumberDec}(k)$ performance profile, the BB variants clearly outperform the constant-step variants. In particular, ALG2-BB-St1 achieves the largest value of $\pi_s(1)$, approximately $0.95$, indicating that it is the most efficient solver for about $95\%$ of the test instances. The corresponding values of $\pi_s(1)$ for ALG2-BB-St3, ALG2-BB-St2, and ALG2-BB-St4 are approximately $0.78$, $0.66$, and $0.02$, respectively. In contrast, each constant-step variant attains $\pi_s(1)\approx0.02$, highlighting the substantial efficiency gains provided by the Barzilai--Borwein step-size strategy.
Regarding robustness, all algorithms except ALG2-BB-St2 successfully solved every test instance. Consequently, with the exception of ALG2-BB-St2, the methods exhibit essentially identical robustness on the well-conditioned test set.
{The CPU-time performance profiles further confirm the effectiveness of the Barzilai--Borwein step-size strategy. ALG2-BB-St1 is the fastest implementation on approximately $37\%$ of the test instances, followed by ALG2-BB-St3 ($32\%$) and ALG2-BB-St2 ($25\%$). In contrast, none of the constant-step variants is the fastest on more than $2\%$ of the test instances. Therefore, although the execution-time profiles are closer than those based on $\texttt{NumberDec}(k)$, the BB variants consistently exhibit superior computational efficiency.}

\begin{figure}[htbp]
\centering
\begin{subfigure}[b]{0.8\linewidth}
    \centering
    \includegraphics[width=\linewidth, height=0.32\textheight, keepaspectratio]{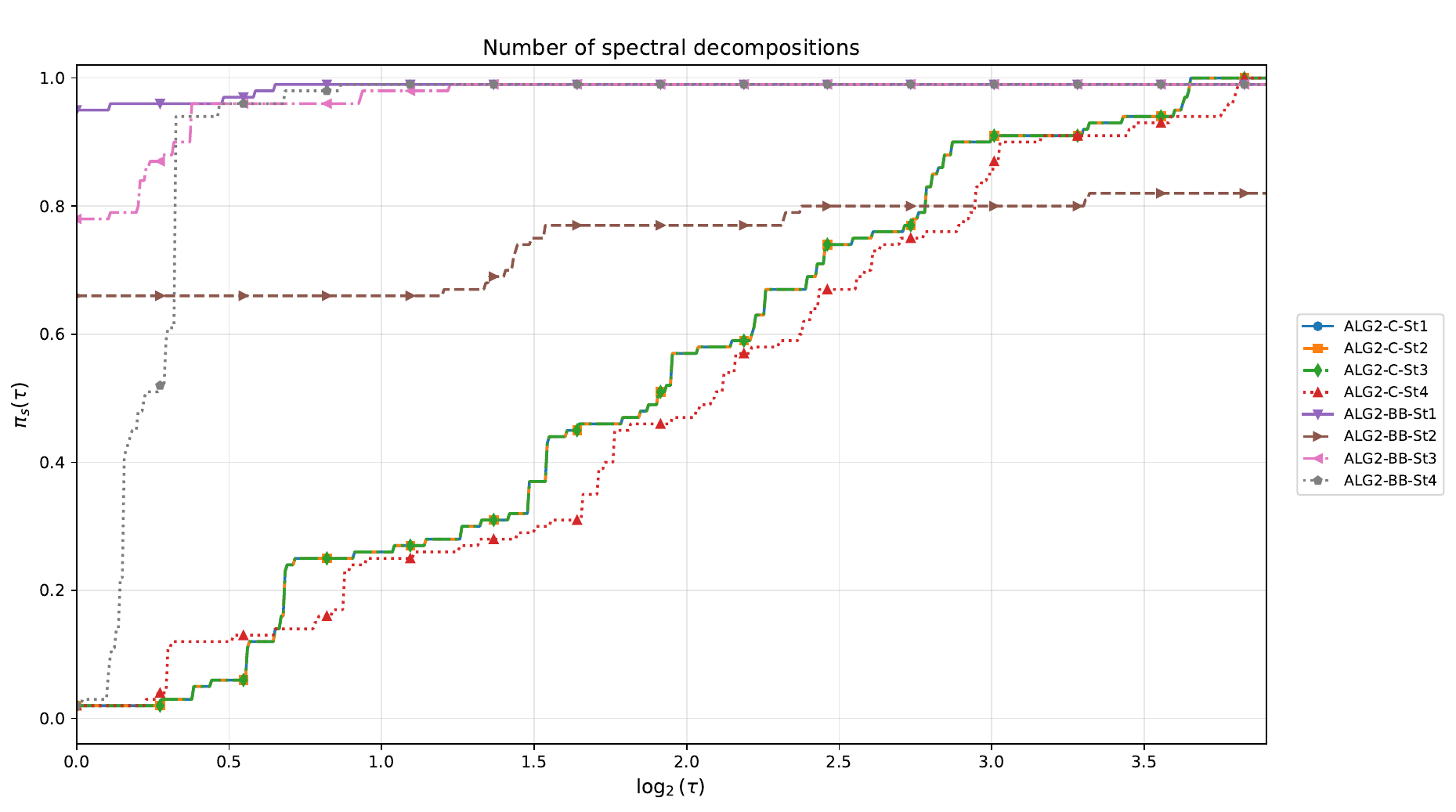}
    \caption{$\texttt{NumberDec}(k)$.}
    \label{fig:karcher-well-nd}
\end{subfigure}

\vspace{0.8em}

\begin{subfigure}[b]{0.8\linewidth}
    \centering
    \includegraphics[width=\linewidth, height=0.32\textheight, keepaspectratio]{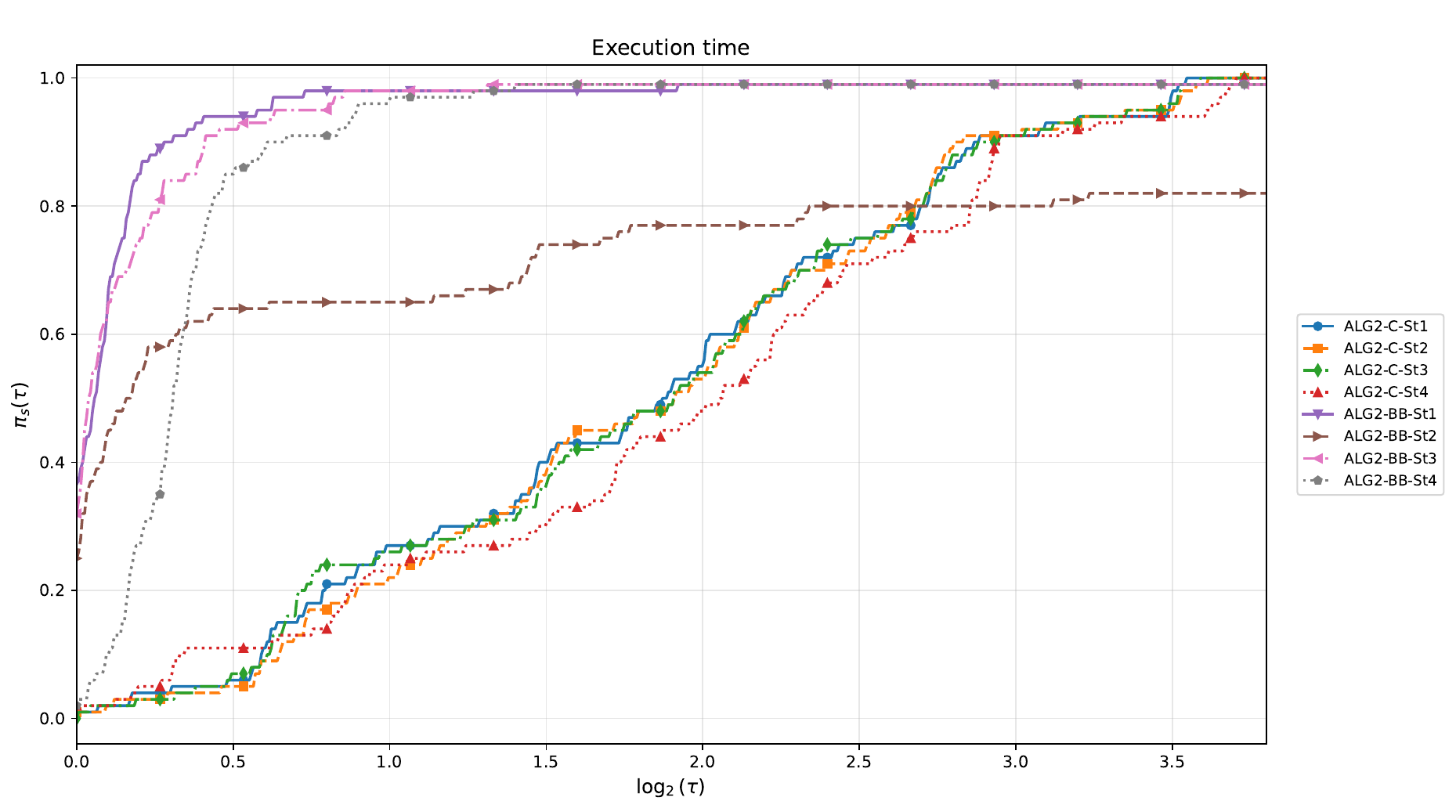}
    \caption{CPU time.}
    \label{fig:karcher-well-time}
\end{subfigure}

\caption{Performance profiles for the constrained Karcher mean problem on the well-conditioned test set.}
\label{fig:karcher-well-performance-profiles}
\end{figure}

Now, the performance profiles, based on the $\texttt{NumberDec}(k)$ metric and CPU time, for the ill-conditioned test set are presented in Figure~\ref{fig:karcher-ill-performance-profiles}. In general, this test set is
 more challenging than the well-conditioned one.  In terms of
$\texttt{NumberDec}(k)$, the BB variants again outperform the
constant-step ones. In particular, ALG2-BB-St1 achieves the largest value of
$\pi_s(1)$, approximately $0.57$, indicating that it is the most efficient solver
for about $57\%$ of the test instances. The corresponding values of $\pi_s(1)$
for ALG2-BB-St3, ALG2-BB-St4, and ALG2-BB-St2 are approximately $0.40$, $0.27$,
and $0.21$, respectively. In contrast, each constant-step variant attains
$\pi_s(1)\approx 0.04$, showing that the Barzilai--Borwein step-size strategy
also provides substantial efficiency gains in the ill-conditioned regime.
The CPU-time performance profiles show a similar behavior. ALG2-BB-St1
is the fastest implementation on approximately $30\%$ of the test instances,
followed by ALG2-BB-St4 ($26\%$), ALG2-BB-St3 ($22\%$), and ALG2-BB-St2
($11\%$). In contrast, none of the constant-step variants is the fastest on more
than $4\%$ of the test instances. 
Regarding robustness, {ALG2-BB-St1} is the most robust variant,
successfully solving $84$ of the $100$ test instances. The variants
{ALG2-BB-St2}, {ALG2-BB-St3}, and {ALG2-BB-St4}
solve $43$, $76$, and $74$ instances, respectively.  On the other hand, ALG2-C-St1 also exhibits remarkable robustness in this setting, solving 76 instances, whereas ALG2-C-St2, ALG2-C-St3, and ALG2-C-St4 solve 59, 65, and 53 instances, respectively.

\begin{figure}[htbp]
\centering
\begin{subfigure}[b]{0.8\linewidth}
    \centering
    \includegraphics[width=\linewidth, height=0.32\textheight, keepaspectratio]{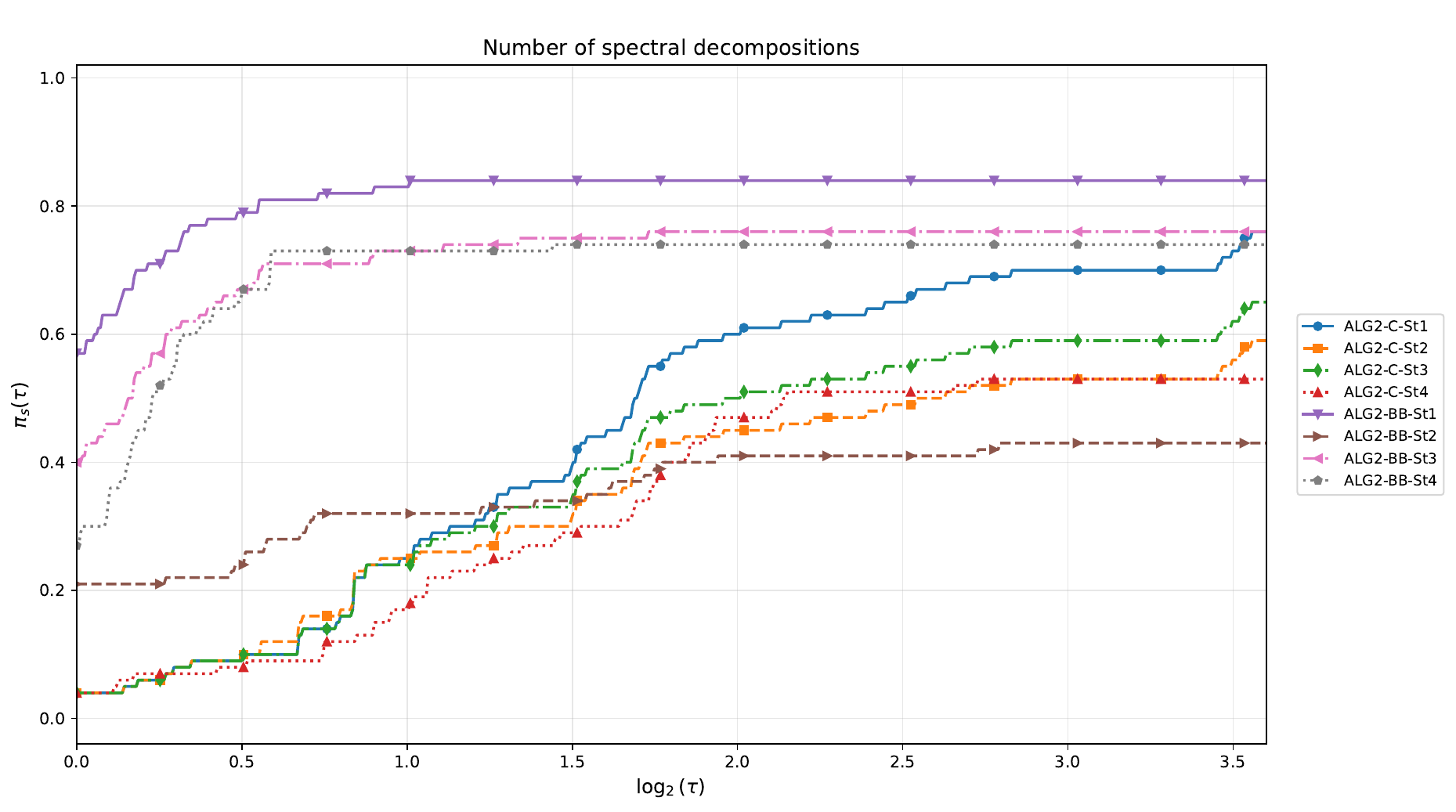}
    \caption{$\texttt{NumberDec}(k)$.}
    \label{fig:karcher-ill-nd}
\end{subfigure}

\vspace{0.8em} 

\begin{subfigure}[b]{0.8\linewidth}
    \centering
    \includegraphics[width=\linewidth, height=0.32\textheight, keepaspectratio]{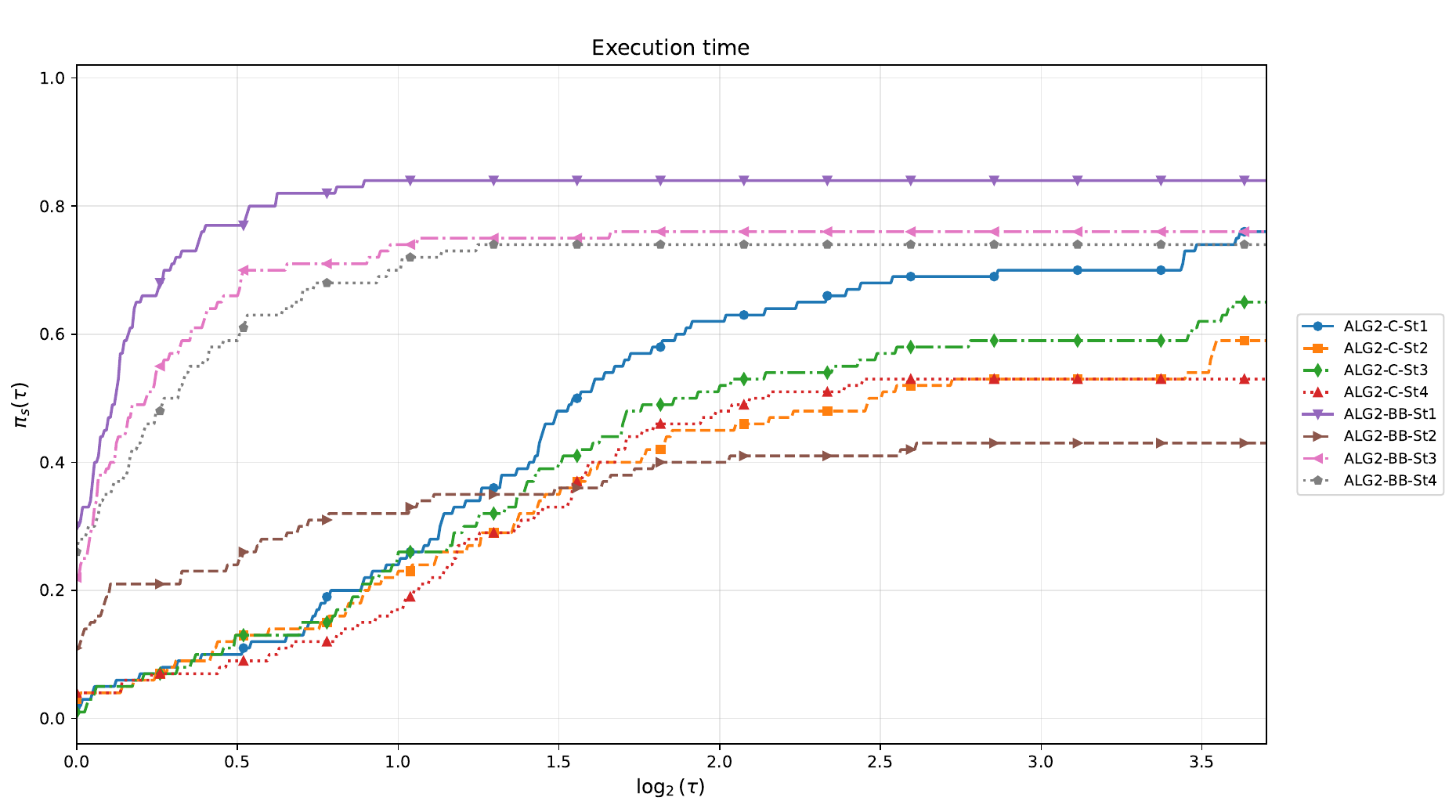}
    \caption{CPU time.}
    \label{fig:karcher-ill-time}
\end{subfigure}

\caption{Performance profiles for the constrained Karcher mean problem on the ill-conditioned test set.}
\label{fig:karcher-ill-performance-profiles}
\end{figure}

In summary, the numerical experiments indicate that the proposed
Barzilai--Borwein step-size strategy consistently improves the practical
performance of Algorithm~2. While all BB variants outperform their
constant-step counterparts in terms of efficiency, ALG2-BB-St1 provides the
best overall performance by achieving the most favorable trade-off between
efficiency and robustness over both well-conditioned and ill-conditioned test
sets.

%---------------------------------------------------------------------------------------------------------
\section{Conclusion}

In this paper, we analyzed two projected gradient schemes for constrained smooth optimization on Hadamard manifolds: a constant-stepsize scheme and a backtracking scheme. For the constant-stepsize method, under Lipschitz continuity of the Riemannian gradient, we established sufficient descent, stationarity of every accumulation point, asymptotic regularity, and an iteration-complexity bound of order \(O(1/\sqrt{N})\) for a natural projection-based stationarity measure. For the backtracking method, we proved well-definedness and stationarity of every accumulation point without requiring Lipschitz continuity of the Riemannian gradient or compactness of the feasible set, provided that the trial line-search stepsizes are uniformly bounded away from zero. Under Lipschitz continuity of the Riemannian gradient, we also obtained an \(O(1/\sqrt{N})\) complexity bound and asymptotic regularity for the backtracking scheme. Compactness assumptions are needed only to guarantee the existence of accumulation points. Numerical experiments on constrained Karcher mean problems over the manifold of symmetric positive definite matrices illustrate the practical behavior of the methods.

The results complement the existing projected- and proximal-gradient literature by extending the direct analysis of projected-direction backtracking schemes from hyperbolic space forms to general Hadamard manifolds, while avoiding compactness or coercivity assumptions in the cluster-point stationarity analysis. Natural directions for future research include establishing convergence of the whole sequence under weaker assumptions, obtaining sharper complexity bounds, and relaxing the uniform lower-bound condition imposed on the trial line-search stepsizes.

\bibliographystyle{plain}
\bibliography{GradProjMet_abbreviated}

\end{document}